\documentclass[12pt,reqno]{amsart}
\usepackage[colorlinks=true, pdfstartview=FitV, linkcolor=blue, citecolor=blue, urlcolor=blue]{hyperref}
\usepackage{amssymb,amsmath, amscd}
\usepackage{ragged2e}
\usepackage{times, verbatim}
\usepackage{mathdots}
\usepackage{array}
\usepackage{graphicx}
\usepackage[english]{babel}
 \usepackage[usenames, dvipsnames]{color}
\usepackage{amsmath,amssymb,amsfonts}
\usepackage{enumerate, enumitem}
\usepackage{MnSymbol}
\usepackage{anysize}
\usepackage{enumitem}
\usepackage{bigints}
\usepackage{braket}
\usepackage[english]{babel}
\usepackage{blindtext}
\usepackage{mathtools}
\usepackage{graphicx}
\usepackage{tikz-cd}
\usepackage[a4paper,verbose]{geometry}
\marginsize{2.5cm}{2.5cm}{2.5cm}{2.5cm}
\input xy
\xyoption{all}
\usepackage{pb-diagram}
\usepackage[all]{xy}
\input xy
\xyoption{all}
\usepackage{scalerel}

\DeclareFontFamily{OT1}{rsfs}{}
\DeclareFontShape{OT1}{rsfs}{n}{it}{<-> rsfs10}{}
\DeclareMathAlphabet{\mathscr}{OT1}{rsfs}{n}{it}
\DeclareRobustCommand\longtwoheadrightarrow
     {\relbar\joinrel\twoheadrightarrow}

\numberwithin{equation}{section}% reset equation counter for sections
\numberwithin{equation}{subsection}
\renewcommand*{\theequation}{%
  \ifnum\value{subsection}=0 %
    \thesection
  \else
    \thesubsection
  \fi
  .\arabic{equation}%
}

\newtheorem{thm}[equation]{Theorem}% This will reset theorem numbering at each subsection level
\newtheorem{prop}[equation]{Proposition}
\newtheorem{corollary}[equation]{Corollary}
\newtheorem{conj}[equation]{Conjecture}
\newtheorem{lemma}[equation]{Lemma}
\newtheorem{defn}[equation]{Definition}

\newtheorem{remark}[equation]{Remark}
\newtheorem{Example}[equation]{Example}

\theoremstyle{definition}
\newtheorem*{funding}{Funding}

\theoremstyle{remark}
\newtheorem*{acknowledgements}{Acknowledgments}

\begin{document}
\theoremstyle{plain}

\makeatletter
\DeclareRobustCommand{\bigboxplus@}{%
  \mathop{\vphantom{\sum@}\scaleobj{0.8}{\scalerel*{\boxplus}{\sum}}}%
}
\newcommand{\bigboxplus}{\DOTSB\bigboxplus@\slimits@}
\DeclareRobustCommand{\bigboxtimes@}{%
  \mathop{\vphantom{\sum@}\scaleobj{0.8}{\scalerel*{\boxtimes}{\sum}}}%
}
\newcommand{\bigboxtimes}{\DOTSB\bigboxtimes@\slimits@}
\makeatother

\newcommand{\Hecke}{\mathcal{H}}
\newcommand{\Liea}{\mathfrak{a}}
\newcommand{\Cmg}{C_{\mathrm{mg}}}
\newcommand{\Cinftyumg}{C^{\infty}_{\mathrm{umg}}}
\newcommand{\Cfd}{C_{\mathrm{fd}}}
\newcommand{\Cinftyfd}{C^{\infty}_{\mathrm{ufd}}}
\newcommand{\sspace}{\Gamma \backslash G}
\newcommand{\PP}{\mathcal{P}}
\newcommand{\bfP}{\mathbf{P}}
\newcommand{\bfQ}{\mathbf{Q}}
\newcommand{\Siegel}{\mathfrak{S}}
\newcommand{\g}{\mathfrak{g}}
\newcommand{\A}{\mathbb{A}}
\newcommand{\Q}{\mathbb{Q}}
\newcommand{\Gm}{\mathbb{G}_m}
\newcommand{\Nm}{\mathbb{N}m}
\newcommand{\ii}{\mathfrak{i}}
\newcommand{\II}{\mathfrak{I}}

\newcommand{\kk}{\mathfrak{k}}
\newcommand{\nn}{\mathfrak{n}}
\newcommand{\tF}{\widetilde{F}}
\newcommand{\p}{\mathfrak{p}}
\newcommand{\m}{\mathfrak{m}}
\newcommand{\bb}{\mathfrak{b}}
\newcommand{\Ad}{{\rm Ad}\,}
\newcommand{\ttt}{\mathfrak{t}}
\newcommand{\frakt}{\mathfrak{t}}
\newcommand{\U}{\mathcal{U}}
\newcommand{\Z}{\mathbb{Z}}
\newcommand{\bfG}{\mathbf{G}}
\newcommand{\bfT}{\mathbf{T}}
\newcommand{\R}{\mathbb{R}}
\newcommand{\ST}{\mathbb{S}}
\newcommand{\h}{\mathfrak{h}}
\newcommand{\bC}{\mathbb{C}}
\newcommand{\C}{\mathbb{C}}
\newcommand{\N}{\mathbb{N}}
\newcommand{\qH}{\mathbb {H}}
\newcommand{\temp}{{\rm temp}}
\newcommand{\Hom}{{\rm Hom}}
\newcommand{\Aut}{{\rm Aut}}
\newcommand{\rk}{{\rm rk}}
\newcommand{\Ext}{{\rm Ext}}
\newcommand{\End}{{\rm End}\,}
\newcommand{\Ind}{{\rm Ind}}
\newcommand{\ind}{{\rm ind}}
\newcommand{\Irr}{{\rm Irr}}
\def\circG{{\,^\circ G}}
\def\M{{\rm M}}
\def\diag{{\rm diag}}
\def\Ad{{\rm Ad}}
\def\As{{\rm As}}
\def\wG{{\widehat G}}
\def\G{{\rm G}}
\def\SL{{\rm SL}}
\def\PSL{{\rm PSL}}
\def\GSp{{\rm GSp}}
\def\PGSp{{\rm PGSp}}
\def\Sp{{\rm Sp}}
\def\St{{\rm St}}
\def\GU{{\rm GU}}
\def\SU{{\rm SU}}
\def\U{{\rm U}}
\def\GO{{\rm GO}}
\def\GL{{\rm GL}}
\def\PGL{{\rm PGL}}
\def\GSO{{\rm GSO}}
\def\GSpin{{\rm GSpin}}
\def\GSp{{\rm GSp}}

\newcommand{\Ssp}{\mathcal S}
\newcommand{\W}{\mathcal W}

\newcommand{\Sym}{\operatorname{Sym}}
\newcommand{\wh}{\widehat\otimes}
\newcommand{\detabs}{|\det|}
\newcommand{\jet}{\operatorname{jet}}
\newcommand{\RS}{\operatorname{RS}}

\def\Gal{{\rm Gal}}
\def\SO{{\rm SO}}
\def\O{{\rm  O}}
\def\Sym{{\rm Sym}}
\def\sym{{\rm sym}}
\def\St{{\rm St}}
\def\Sp{{\rm Sp}}
\def\tr{{\rm tr\,}}
\def\ad{{\rm ad\, }}
\def\Ad{{\rm Ad\, }}
\def\rank{{\rm rank\,}}

\def\Ext{{\rm Ext}}
\def\Hom{{\rm Hom}}
\def\Alg{{\rm Alg}}
\def\GL{{\rm GL}}
\def\SO{{\rm SO}}
\def\G{{\rm G}}
\def\U{{\rm U}}
\def\St{{\rm St}}
\def\Wh{{\rm Wh}}
\def\RS{{\rm RS}}
\def\ind{{\rm ind}}
\def\Ind{{\rm Ind}}
\def\csupp{{\rm csupp}}

%\subjclass{Primary 11F70; Secondary 22E55}
\title[Exceptional poles for irreducible generic representations of $\operatorname{GL}_n(\mathbb{R})$]{Exceptional poles of Archimedean Rankin--Selberg $L$-functions for irreducible generic representations of $\operatorname{GL}_n(\mathbb{R})$}
\author[Yeongseong Jo, Santosh Nadimpalli and Akash Yadav]{Yeongseong Jo$^{\dagger}$, Santosh Nadimpalli$^{\dagger}$ and Akash Yadav$^{\dagger}$}
\thanks{$^{\dagger}$All authors contributed equally and are listed alphabetically.}
\address{(Y. Jo) Department of Mathematics Education / Human-Centered Artificial Intelligence Research Institute, Ewha Womans University, Seoul 03760, Republic of Korea}
\email{yeongseong.jo@ewha.ac.kr, yeongseongjo@outlook.com}
\address{(S. Nadimpalli) Department of Mathematics and Statistics, Indian Institute of Technology Kanpur, Uttar Pradesh 208016, India}
\email{nsantosh@iitk.ac.in, nvrnsantosh@gmail.com}
\address{(A. Yadav) Human-Centered Artificial Intelligence Research Institute, Ewha Womans University, Seoul 03760, Republic of Korea}
\email{akaschampion@ewha.ac.kr, akaschampion@gmail.com}


\begin{abstract}
For irreducible generic representations $\pi_1$ and $\pi_2$ of
$\GL_n(\mathbb R)$, we prove that the notions of exceptional pole of
type~$1$ and type~$2$ coincide at every level. When both representations
are in general position, we use this identification to express the
Rankin--Selberg $L$-function $L(s,\pi_1\times\pi_2)$ in terms of the
exceptional $L$-factors attached to the irreducible constituents of
their derivatives.
\end{abstract}
\maketitle

\begin{section}{Introduction}

Let \(F\) be a local field of characteristic \(0\), and let \(\pi_1\) and \(\pi_2\) be irreducible generic representations of \(\GL_n(F)\).
When \(F\) is non-archimedean, Cogdell and Piatetski-Shapiro~\cite{CPS2017} developed a method for computing the local Rankin--Selberg \(L\)-functions \(L(s,\pi_1 \times \pi_2)\) in terms of \(L\)-functions attached to supercuspidal representations.
They use the theory of Bernstein--Zelevinsky derivatives~\cite{BZ1976,BZ1977,Zel1980}
and analyze exceptional poles in detail.
More precisely, they express \(L(s,\pi_1 \times \pi_2)\) in terms of exceptional \(L\)-factors attached to derivatives of \(\pi_1\) and \(\pi_2\).
In this setting, the Schwartz space \(\mathcal{S}(F^n)\) admits a one-step filtration given by the subspace of functions vanishing at \(0\).
Consequently, exceptional poles have no nontrivial level structure
(see, e.g., \cite{CPS2017,Mat2015}).
The corresponding notion in the archimedean setting is that of an
exceptional pole with level~\(0\) (see Section~\ref{s2}).

\par
Exceptional poles have been extensively studied when \(F\) is non-archimedean (see, e.g., \cite{CPS2017,Mat2015,Jo2018}), whereas the archimedean case remains far less understood.
A first step in this direction was taken by Chang and Cogdell~\cite{CC1999}, who studied the \(\eta\)-homology of representations, where \(\eta\) denotes the nilradical of a suitable parabolic subalgebra.
They showed that the \(\eta\)-homology of irreducible generic representations in general position is completely reducible.
For the definition of \emph{general position}, we refer the reader to Section~\ref{s2}.

For \(F=\mathbb R\), Chai~\cite{Chai2015} used this result together
with \cite[Definition~3.0.2]{AGS2015} to define a notion of
archimedean derivative via a suitable Whittaker quotient of
\(\eta\)-homology. He also formulated a
representation-theoretic notion of exceptional poles of type~\(2\)
with level \(m\), where \(m\) is a nonnegative integer.
Exceptional poles of type~\(1\) with level \(m\), on the other hand, are defined via Rankin--Selberg integrals \cite{JPSS1983}.
Unlike the non-archimedean case, the level structure in the archimedean setting arises from the infinite filtration of the Schwartz space.
Chai~\cite{Chai2015} conjectured that the two notions of exceptional poles with level \(m\) coincide.
The conjecture for \(m=0\) and \(s_0=0\) was proved by the third-named
author~\cite{Yad2024}: in this case, the type~\(2\) condition is precisely that \(\pi_1\widehat{\otimes}\pi_2\) be \(\GL_n(\mathbb{R})\)-distinguished.
Here \(\widehat{\otimes}\) denotes the completed projective tensor product of smooth Fréchet representations.

\par
Our main result proves Chai's conjecture for all pairs of irreducible
generic representations of \(\GL_n(\mathbb R)\).

\begin{thm}\label{1.1}
Let \(\pi_1\) and \(\pi_2\) be irreducible generic representations of
\(\GL_n(\mathbb{R})\). Let
\(m\in\mathbb{Z}_{\geq0}\) and \(s_0\in\mathbb{C}\). Then
\(s=s_0\) is an exceptional pole of type~\(1\) with level \(m\) for
\((\pi_1,\pi_2)\) if and only if it is an exceptional pole of type~\(2\)
with level \(m\).
\end{thm}

\par
The implication from type~\(1\) to type~\(2\) follows directly from the
definitions, whereas the converse is the main difficulty in the proof
of Theorem~\ref{1.1}. We first study a filtration of
$\pi_2\widehat{\otimes}\operatorname{Sym}^m(\mathbb C^n)$
whose successive quotients are induced from algebraic twists of the
inducing blocks of $\pi_2$. An analysis of the infinitesimal characters
of these quotients shows that every exceptional pole $s=s_0$
of type~\(2\) with level $m$ is a pole of
$L(s,\pi_1\times\pi_2)$.

Next, composing a nonzero type~\(2\) functional of level $m$ with the
homogeneous degree-$m$ Taylor map in the Schwartz variable produces
a continuous equivariant trilinear form that vanishes on
$\mathcal S^{m+1}$ in the Schwartz variable but not identically on
$\mathcal S^m$. The normalized Rankin--Selberg functional
$\Lambda_{s_0}$ is nonzero and satisfies the same equivariance
property. The uniqueness theorem of Chen and Sun
\cite[Theorem~B]{CS2015} therefore identifies these two forms up to
a nonzero scalar. Since the leading Laurent coefficient $B_{s_0}$
of the family of Rankin--Selberg integrals is also a nonzero scalar
multiple of $\Lambda_{s_0}$, it has the same vanishing properties.
Hence $s=s_0$ is an exceptional pole of type~\(1\) with level $m$.

\par
As an application of Theorem~\ref{1.1},  when $\pi_1$ and $\pi_2$ are in general position, we show that the inverse of the archimedean Rankin--Selberg \(L\)-function is the least common multiple
of the inverse exceptional \(L\)-factors attached to the irreducible constituents of the derivatives of \(\pi_1\) and \(\pi_2\). Analogous results have been established in the non-archimedean setting for the Rankin--Selberg $L$-function \cite[Theorem~2.1]{CPS2017}, the exterior square $L$-function \cite[Theorems~2.7 and~2.13]{Jo2018}, and the Bump--Friedberg $L$-function \cite[Proposition~4.5]{Mat2015}. 
%In the non-archimedean context, Theorem 5.8 of \cite{Chai2015} leads the result, which
%was used to lay the groundwork for computation of local $L$-functions  \cite{CPS2017,Mat2015,Jo2018}. 
%On the other hand, we employ the factorization of archimedean $L$-functions attached to Langlands parameters (cf. Proposition \ref{pole}) instead of invoking \cite[Theorem 5.8]{Chai2015}.
To the best of our knowledge, this provides the first such description in the archimedean case. 
For definitions of $\operatorname{l.c.m.}$ and \(L_{\mathrm{ex}}\), we refer the reader to Section~\ref{s5}. 

\begin{thm}\label{1.2}
Let \(\pi_1\) and \(\pi_2\) be irreducible generic representations of 
\(\GL_n(\mathbb{R})\) in general position. For each \(k\) with \(0 \le k < n\), write
\[
\pi_1^{(k)} = \bigoplus_i \pi_{1,i}^{(k)}, 
\qquad
\pi_2^{(k)} = \bigoplus_j \pi_{2,j}^{(k)},
\]
where \(\pi_{1,i}^{(k)}\) and \(\pi_{2,j}^{(k)}\) denote the irreducible constituents of the $k$-th derivatives of \(\pi_1\) and \(\pi_2\), respectively. Then
\[
L(s,\pi_1 \times \pi_2)^{-1}
=
\underset{\,0\le k < n,\; i,j}{\operatorname{l.c.m.}}
\left\{
L_{\mathrm{ex}}\!\left(s,\pi_{1,i}^{(k)} \times \pi_{2,j}^{(k)}\right)^{-1}
\right\}.
\]
\end{thm}

In the non-archimedean context,
\cite[Theorem~2.1]{CPS2017} is deduced from
\cite[Proposition~2.3]{CPS2017}. We prove Theorem~\ref{1.2} by comparing pole orders using the
factorization of archimedean $L$-functions attached to Langlands
parameters. The proof does not use \cite[Theorem~5.8]{Chai2015}.

\par
The paper is organized as follows. In Section~\ref{s2}, we fix
notation and recall the necessary definitions. In Section~\ref{s3},
we show that every exceptional pole of type~\(2\) is a pole of the
Rankin--Selberg \(L\)-function. Under the general-position hypothesis,
we also classify these exceptional poles and compute their orders.
In Section~\ref{s4}, we prove Theorem~\ref{1.1}. In Section~\ref{s5},
we define the exceptional \(L\)-factors for pairs in general position,
compute the derivatives explicitly, and prove Theorem~\ref{1.2}.

\end{section}

\begin{section}{Notation and Preliminaries}\label{s2}

Let \(\lvert\>\cdot\>\rvert\) denote the usual absolute value on
\(\mathbb{R}\). For \(n\in\mathbb{N}\), set
\(G_n=\mathrm{GL}_n(\mathbb{R})\), and let \(Z_n\), \(B_n\), and \(N_n\)
denote the subgroups of scalar, upper triangular, and unipotent upper
triangular matrices in \(G_n\), respectively. Set \(K_n=\mathrm{O}(n)\),
and let
\[
e_n=(0,\ldots,0,1)\in\mathbb{R}^n,
\]
viewed as a row vector. Let \(\mathbb{S}^1\) denote the
multiplicative group of complex numbers of absolute value \(1\).

By a representation of \(G_n\), we mean a Casselman--Wallach
representation, that is, a smooth admissible Fr\'echet representation
of moderate growth. We write \(\Irr(G_n)\) for the set of isomorphism
classes of irreducible representations of \(G_n\). We frequently abuse
notation by identifying a representation
\(\pi\in\Irr(G_n)\) with its underlying vector space \(V_\pi\);
accordingly, we may refer to this space as \(\pi\), \(V_\pi\), or
simply \(V\). Given two representations \(\pi_1\) and \(\pi_2\), we
denote their completed projective tensor product by
\(\pi_1\widehat{\otimes}\pi_2\). Throughout, \(\operatorname{Hom}_{G_n}\) denotes the space of
continuous \(G_n\)-equivariant linear maps.

For $\pi\in \Irr(G_n)$, we write $\omega_\pi$ for its central character.
Let $P=P_{n_1,\ldots,n_k}$ be the standard parabolic subgroup of $G_n$
associated to the partition $n=n_1+\cdots+n_k$, and let $P=MU$ be its Levi
decomposition. Then $M$ is of the form
\[
M = G_{n_1}\times \cdots \times G_{n_k}.
\]
Let
\(\mathfrak n_{n_1,\ldots,n_k}:=\operatorname{Lie}(U)\).
Let \(\mathfrak g_n:=\operatorname{Lie}(G_n)\), and write $\mathfrak g_{n,\mathbb C}
:= \mathfrak g_n\otimes_{\mathbb R}\mathbb C
$ for its complexification.

For each $1\le i\le k$, let $\tau_i\in \Irr(G_{n_i})$, and set
$\sigma=\tau_1\boxtimes\cdots\boxtimes\tau_k$, viewed as a representation of
$M$. We denote by
\[
\tau_1 \times \cdots \times \tau_k
\]
the normalized parabolic induction
$\Ind_{P}^{G_n}(\sigma)$. Let $\chi_{\epsilon,s}$ denote the character of $\mathbb{R}^{\times}$ defined by
\[
\chi_{\epsilon,s}(x)=\operatorname{sgn}(x)^{\epsilon}\lvert x\rvert^{s},
\]
where $\epsilon\in\{0,1\}$ and $s\in\mathbb{C}$.

Let $k \ge 2$ be an integer and consider the reducible principal series representation of $G_2$ given by
\[
I(k) := \chi_{\varepsilon_k, -\frac{k-1}{2}} \times \chi_{0, \frac{k-1}{2}},
\]
where $\varepsilon_k\in\{0,1\}$ denotes the parity of $k$. This representation fits into the short exact sequence
\begin{equation}\label{ses:Ik}
0 \longrightarrow
\operatorname{Sym}^{k-2}(\mathbb{C}^2)
\otimes
\left(
\chi_{\varepsilon_k,-\frac{k-2}{2}}\circ\det
\right)
\longrightarrow I(k)
\longrightarrow D_k
\longrightarrow 0.
\end{equation}
Here $D_k$ denotes the discrete series representation of $G_2$ with central character $\chi_{\varepsilon_k,0}$. For $s \in \mathbb{C}$, we denote by $D_{k,s} := D_k \otimes |\det|^s
$ the corresponding twisted representation.

\par
Fix a nontrivial additive character
\(\psi:\mathbb R\to\mathbb S^1\). We use the same symbol for the
generic character of \(N_n\) defined by
\[
\psi(u)
=
\psi\!\left(\sum_{i=1}^{n-1}u_{i,i+1}\right),
\qquad u\in N_n.
\]
A representation \(\pi\in\Irr(G_n)\) is \emph{generic} if there is a
nonzero continuous linear functional \(\lambda:\pi\to\mathbb C\)
such that
\[
\lambda(\pi(u)v)=\psi(u)\lambda(v)
\qquad
(u\in N_n,\ v\in\pi).
\]
We write \(\Irr_{\mathrm{gen}}(G_n)\) for the subset of generic
representations. For \(\pi\in\Irr_{\mathrm{gen}}(G_n)\), let
\(\mathcal W(\pi,\psi)\) denote its Whittaker model, and write
\(\pi^\vee\) for its smooth contragredient.
By \cite{SV1980}, any irreducible generic representation of $G_n$ is isomorphic to a parabolically induced representation of the form
\[
\pi=
\left(\bigtimes_{i=1}^{r}\chi_{\epsilon_i,s_i}\right)
\times
\left(\bigtimes_{j=1}^{l}D_{k_j,t_j}\right),
\qquad n=r+2l,
\]
where $\epsilon_i\in\{0,1\}$, $k_j\ge 2$, and $s_i,t_j\in\mathbb{C}$, satisfying the following conditions:
\begin{enumerate}
\item $s_i-s_{i'} \notin |\epsilon_i-\epsilon_{i'}|-1+2\mathbb{Z}_{>0}$ for all $i\neq i'$;
\item $s_i-t_j \notin \pm\left(\dfrac{k_j-1}{2}+\mathbb{Z}_{>0}\right)$ for all $1\le i\le r$ and $1\le j\le l$;
\item $t_j-t_{j'} \notin \dfrac{|k_j-k_{j'}|}{2}+\mathbb{Z}_{>0}$ for all $j\neq j'$.
\end{enumerate}
Conversely, a normalized parabolic induction of the displayed form
is irreducible and generic whenever these conditions hold.

\par
Next, we define the notion of general position for parabolically induced representations, following \cite{CC1999}. Let
\[
\pi=
\left(\bigtimes_{i=1}^{r}\chi_{\epsilon_i,s_i}\right)
\times
\left(\bigtimes_{j=1}^{l}D_{k_j,t_j}\right),
\qquad n=r+2l,
\]
where $\epsilon_i\in\{0,1\}$, $k_j\ge 2$, and $s_i,t_j\in\mathbb{C}$.
We say that \(\pi\) is in \emph{general position} if
\[
\begin{aligned}
s_i&\notin\mathbb Z
&& (1\leq i\leq r),\\
t_j&\notin\tfrac12\mathbb Z
&& (1\leq j\leq l),\\
s_i-s_{i'}&\notin\mathbb Z
&& (1\leq i\neq i'\leq r),\\
s_i-t_j&\notin\tfrac12\mathbb Z
&& (1\leq i\leq r,\ 1\leq j\leq l),\\
t_j-t_{j'}&\notin\tfrac12\mathbb Z
&& (1\leq j\neq j'\leq l).
\end{aligned}
\]
Under these hypotheses, \(\pi\) is irreducible and generic.

For \(F\in\{\mathbb R,\mathbb C\}\), define
\[
\Gamma_F(s)
=
\begin{cases}
\pi^{-s/2}\Gamma\!\left(\dfrac{s}{2}\right),
& F=\mathbb R,\\[1mm]
2(2\pi)^{-s}\Gamma(s),
& F=\mathbb C.
\end{cases}
\]
Here \(\Gamma(s)\) denotes the usual Gamma function.

Let \(W_{\mathbb R}\) be the Weil group of \(\mathbb R\), and let
\(\phi:W_{\mathbb R}\to\GL(V)\) be an admissible finite-dimensional
complex representation. To \(\phi\) one associates a local
\(L\)-factor \(L(s,\phi)\) as in~\cite{Tate1979}; it is a finite
product of factors of the form
\(\Gamma_{\mathbb R}(s+a)\), with \(a\in\mathbb C\).
 If $\pi_1$ and $\pi_2$ are irreducible representations of
$G_n$, and if
$\phi_1,\phi_2:W_\mathbb{R}\to\GL_n(\mathbb{C})$ are the admissible representations
associated to $\pi_1$ and $\pi_2$ by the local Langlands correspondence
\cite{Kna1994}, we define
\[
L(s,\pi_1\times\pi_2):=L(s,\phi_1\otimes\phi_2),
\]
and call it the \emph{Rankin--Selberg $L$-function} of $\pi_1$ and $\pi_2$.

\par
The notion of archimedean exceptional poles was introduced by Chai in \cite{Chai2015}. We begin with the representation-theoretic definition of exceptional poles of type~\(2\).
Let $\pi_1,\>\pi_2\in\Irr_{\mathrm{gen}}(G_n)$. 
Let
\(\mathcal S=\mathcal S(\mathbb R^n)\)
denote the Schwartz space on \(\mathbb R^n\). We let \(G_n\) act on
\(\mathcal S\) by
\begin{equation}\label{eq:row-action}
(g\cdot\Phi)(x)=\Phi(xg).
\end{equation}
The space $\mathcal{S}$
admits a canonical decreasing filtration
\begin{equation}
\label{schwartz-decreasing}
\mathcal{S}=\mathcal{S}^0\supseteq \mathcal{S}^1\supseteq \mathcal{S}^2\supseteq\cdots,
\end{equation}
where, for each integer $m\ge 0$, the subspace $\mathcal{S}^m$ consists of functions vanishing 
at the origin to order at least $m$; explicitly,
\[
\mathcal{S}^m=\bigl\{\,\Phi\in\mathcal{S}\,|\, D^\alpha\Phi(0)=0 
\ \text{for all multiindices }\alpha \text{ with }|\alpha|<m\,\bigr\}.
\]
For every \(m\geq0\), taking the homogeneous degree-\(m\) Taylor term
at the origin induces a \(G_n\)-equivariant isomorphism
\[
\mathcal S^m/\mathcal S^{m+1}
\cong
\operatorname{Sym}^m(\mathbb C^n),
\]
where the right-hand side is realized as the space of homogeneous
polynomials of degree \(m\) in a row variable, with action
\[
(g\cdot P)(x)=P(xg).
\]
Let $\rho=\Sym^m(\mathbb C^n)$
be a finite-dimensional representation of $G_n$.

\begin{defn}
Let \(m\in\mathbb Z_{\geq0}\) and \(s_0\in\mathbb C\). We say that
\(s=s_0\) is an \emph{exceptional pole of type~\(2\) with level \(m\)}
for \((\pi_1,\pi_2)\) if
\[
\operatorname{Hom}_{G_n}\!\left(
\pi_1\widehat{\otimes}\pi_2\widehat{\otimes}\rho,
\lvert\det\rvert^{-s_0}
\right)
\neq0.
\]
\end{defn}

Exceptional poles of type~\(1\) are defined using the Rankin--Selberg integrals \cite{JPSS1983}.  
For
\[
W \in \mathcal{W}(\pi_1,\psi), \quad
W' \in \mathcal{W}(\pi_2,\psi^{-1}), \quad
\Phi \in \mathcal{S},
\]
the Rankin--Selberg integrals, whenever convergent, are given by
\[
I(s,W,W',\Phi)
=
\int_{N_n\backslash G_n}
W(g)W'(g)\Phi(e_ng)\lvert\det g\rvert^s\,dg,
\qquad s\in\mathbb C,
\]
where \(dg\) is a fixed right-invariant measure on
\(N_n\backslash G_n\). These integrals converge absolutely for
\(\operatorname{Re}(s)\) sufficiently large and admit meromorphic
continuation to the entire complex plane~\cite{Jac2009}.
Jacquet~\cite[Theorem~2.3]{Jac2009} further proved that
\[
\frac{I(s,W,W',\Phi)}
     {L(s,\pi_1\times\pi_2)}
\]
is entire and of finite order in vertical strips.

\par
Suppose that \(s=s_0\) is a pole of at least one Rankin--Selberg
integral, and let \(d\) be the maximal pole order at \(s_0\) among all
choices of \(W\), \(W'\), and \(\Phi\). For each choice of data, write
\[
I(s,W,W',\Phi)
=
\frac{B_{s_0}(W,W',\Phi)}{(s-s_0)^d}
+\cdots,
\]
where \(B_{s_0}(W,W',\Phi)\) is the coefficient of
\((s-s_0)^{-d}\), understood to be zero if the corresponding integral
has pole order less than \(d\). Then \(B_{s_0}\) is a nonzero
continuous trilinear form \cite{CPS2004} on
\[
\mathcal W(\pi_1,\psi)
\times
\mathcal W(\pi_2,\psi^{-1})
\times
\mathcal S
\]
satisfying
\begin{equation}\label{trilinear-invariance}
B_{s_0}(g\cdot W,g\cdot W',g\cdot\Phi)
=
\lvert\det g\rvert^{-s_0}
B_{s_0}(W,W',\Phi).
\end{equation}
Here \(G_n\) acts on the Whittaker models by right translation,
\[
(g\cdot W)(h)=W(hg),
\]
and on \(\mathcal S\) by~\eqref{eq:row-action}.

\begin{defn}
Let \(m\in\mathbb Z_{\geq0}\), and suppose that \(s=s_0\) is a pole of
the family of Rankin--Selberg integrals. We say that \(s=s_0\) is an
\emph{exceptional pole of type~\(1\) with level \(m\)} for
\((\pi_1,\pi_2)\) if the corresponding 
trilinear form $B_{s_0}$ vanishes on $\mathcal{S}^{m+1}$ but is not identically zero on 
$\mathcal{S}^{m}$. 
\end{defn}

%\begin{remark}
    %In this case, we also say that $s_0$ is an exceptional pole of the Rankin--Selberg integral $I(s,W,W',\Phi)$ or of the $L$-function $L(s,\pi_1\times\pi_2)$.
%\end{remark}

Chai~\cite{Chai2015} conjectured the following.

\begin{conj}\label{conjecture-main}
Let \(\pi_1,\pi_2\in\Irr_{\mathrm{gen}}(G_n)\),
\(m\in\mathbb Z_{\geq0}\), and \(s_0\in\mathbb C\). Then \(s=s_0\) is
an exceptional pole of type~\(1\) with level \(m\) for
\((\pi_1,\pi_2)\) if and only if it is an exceptional pole of type~\(2\)
with level \(m\).
\end{conj}

\begin{remark}
The implication from type~\(1\) to type~\(2\) follows directly from
the definitions. The converse is the substantive content of
Conjecture~\ref{conjecture-main}.
\end{remark}

\section{Realization of Exceptional Poles of Type 2 as Poles}\label{s3}

We begin this section with the following lemma.

\begin{lemma}\label{lem:primary-decomposition}
Let \(X\) be a finite-length Casselman--Wallach representation of
\(G_n\). For an infinitesimal character \(\zeta\), let
\(\mathfrak m_\zeta\) denote the corresponding maximal ideal of
\(Z(\mathfrak g_{n,\mathbb C})\), and set
\[
X[\zeta]
=
\left\{
x\in X\,|\,
\mathfrak m_\zeta^N\cdot x=0
\text{ for some }N\geq1
\right\}.
\]
Then:
\begin{enumerate}
\item $X$ decomposes as a finite topological direct sum
\[
X = \bigoplus_{\zeta} X[\zeta].
\]
\item Each projection $X \to X[\zeta]$ is continuous and $G_n$-equivariant.
\item The functor
\[
X \longmapsto X[\zeta]
\]
is exact on the category of finite-length Casselman--Wallach representations.
\end{enumerate}
\end{lemma}

\begin{proof}
Let $X_{K_n}$ denote the Harish--Chandra module of $K_n$-finite vectors in $X$. 
Since $X$ has finite length, so does $X_{K_n}$, and hence only finitely many 
maximal ideals of $Z(\mathfrak{g}_{n,\mathbb C})$ occur in it. 
The usual primary decomposition, together with the Chinese remainder theorem, 
yields an algebraic direct sum decomposition
\[
X_{K_n} = \bigoplus_{\zeta} X_{K_n}[\zeta],
\]
where each summand is the corresponding generalized $\zeta$-eigenspace. 
Moreover, the functor $M \mapsto M[\zeta]$ is exact on the category of 
finite-length $(\mathfrak{g}_{n,\mathbb C}, K_n)$-modules.

Finally, the Casselman--Wallach globalization functor is an exact equivalence 
from the category of finite-length $(\mathfrak{g}_{n,\mathbb C}, K_n)$-modules 
to the category of finite-length Casselman--Wallach representations. 
Applying this equivalence yields continuous $G_n$-equivariant
projections, preserves exactness, and gives the asserted
topological direct sum
\[
X = \bigoplus_{\zeta} X[\zeta].\qedhere
\]
\end{proof}

\begin{subsection}   
{Algebraic Twist of Discrete Series Representations}
\label{discrete-twist}

Let $k\geq 2$. Recall that we have the short exact sequence
\eqref{ses:Ik}
\[
0 \longrightarrow
\operatorname{Sym}^{k-2}(\mathbb{C}^2)
\otimes
\left(
\chi_{\varepsilon_k,-\frac{k-2}{2}}\circ\det
\right)
\longrightarrow
\chi_{\varepsilon_k,-\frac{k-1}{2}}
\times
\chi_{0,\frac{k-1}{2}}
\longrightarrow
D_k
\longrightarrow 0,
\]
where, for an integer $r$, we denote by
$\varepsilon_r\in\{0,1\}$ the parity of $r$.

Let $m\geq 0$ be an integer. Tensoring this sequence with
$\operatorname{Sym}^m(\mathbb{C}^2)$, we obtain another short exact
sequence
\[
\begin{aligned}
0 \longrightarrow\;&
\operatorname{Sym}^{k-2}(\mathbb{C}^2)
\otimes
\operatorname{Sym}^m(\mathbb{C}^2)
\otimes
\left(
\chi_{\varepsilon_k,-\frac{k-2}{2}}\circ\det
\right)\\
\longrightarrow\;&
\operatorname{Sym}^m(\mathbb{C}^2)
\otimes
\left(
\chi_{\varepsilon_k,-\frac{k-1}{2}}
\times
\chi_{0,\frac{k-1}{2}}
\right)\\
\longrightarrow\;&
\operatorname{Sym}^m(\mathbb{C}^2)\otimes D_k
\longrightarrow 0.
\end{aligned}
\]

As a $B_2$-module, $\operatorname{Sym}^m(\mathbb{C}^2)$ admits a
$B_2$-stable filtration
\[
0=F^0\subseteq F^1\subseteq\cdots\subseteq
F^{m+1}=\operatorname{Sym}^m(\mathbb{C}^2)
\]
whose successive quotients are one-dimensional:
\[
F^i/F^{i-1}\cong\mathbb{C}_{(m+1-i,i-1)},
\qquad 1\leq i\leq m+1,
\]
where $\mathbb{C}_{(a,b)}$ denotes the one-dimensional $B_2$-module
on which $\operatorname{diag}(t_1,t_2)$ acts by $t_1^a t_2^b$.

The tensor identity and exactness of normalized parabolic induction
yield a $G_2$-stable filtration
\[
0=M_0\subseteq M_1\subseteq\cdots\subseteq M_{m+1}
=
\operatorname{Sym}^m(\mathbb{C}^2)
\otimes
\left(
\chi_{\varepsilon_k,-\frac{k-1}{2}}
\times
\chi_{0,\frac{k-1}{2}}
\right),
\]
whose successive quotients are the principal series
\[
P_i\coloneqq M_i/M_{i-1}
\cong
\chi_{\varepsilon_{k+m+1-i},
-\frac{k-1}{2}+m+1-i}
\times
\chi_{\varepsilon_{i-1},
\frac{k-1}{2}+i-1},
\qquad 1\leq i\leq m+1.
\]

Put
\[
r\coloneqq\min(k-2,m)
\]
and let
\[
K\coloneqq
\operatorname{Sym}^{k-2}(\mathbb{C}^2)
\otimes
\operatorname{Sym}^m(\mathbb{C}^2)
\otimes
\left(
\chi_{\varepsilon_k,-\frac{k-2}{2}}\circ\det
\right).
\]
By the Clebsch--Gordan decomposition,
\begin{equation}\label{directsumfd}
    K\cong\bigoplus_{j=0}^{r}S_j,
\end{equation}
where
\[
S_j\coloneqq
\operatorname{Sym}^{k+m-2-2j}(\mathbb{C}^2)
\otimes
\left(
\chi_{\varepsilon_{k+j},
\,j-\frac{k-2}{2}}\circ\det
\right).
\]
The representations \(S_0,\ldots,S_r\) are pairwise nonisomorphic. Hence the decomposition
\eqref{directsumfd} is multiplicity-free, and every \(G_2\)-submodule \(K\cap M_i\) is the
direct sum of a subset of the \(S_j\)'s.
For \(1\leq i\leq m+1\), put
\[
N_i\coloneqq k+2i-m-3.
\]
The condition \(N_i>0\) is equivalent to
\[
i\geq\frac{m+4-k}{2}.
\]
For such an index \(i\), one has
\[
P_i
\cong
I(N_i+1)\otimes
\left(
\chi_{\varepsilon_{i-1},\,\frac m2}\circ\det
\right).
\]
Consequently, we have the short exact sequence
\[
0\longrightarrow\;
\operatorname{Sym}^{N_i-1}(\mathbb C^2)
\otimes
\left(
\chi_{\varepsilon_{k+m+1-i},
\,\frac{m-N_i+1}{2}}\circ\det
\right)\;
\longrightarrow\;
P_i \;
\longrightarrow
D_{N_i+1,\frac m2}
\longrightarrow0.
\]
Here and below, we use the isomorphism
\[
D_q\otimes(\operatorname{sgn}\circ\det)\cong D_q,
\qquad q\geq2,
\]
to suppress a possible common sign twist.

If \(N_i<0\), the opposite rank-one composition series gives
\[
0\longrightarrow\;
D_{|N_i|+1,\frac m2}\;
\longrightarrow
P_i\;
\longrightarrow\;
\operatorname{Sym}^{|N_i|-1}(\mathbb C^2)
\otimes
\left(
\chi_{\varepsilon_{i-1},\,\frac{m-|N_i|+1}{2}}
\circ\det
\right)\;
\longrightarrow0.
\]
If \(N_i=0\), then \(P_i\) is
irreducible and generic.
Under the inclusion $\begin{tikzcd} K\arrow[r, hook] &M_{m+1},\end{tikzcd}$ the summand
$S_j$ occurs in the graded piece $P_{m+1-j}$ as its
finite-dimensional subrepresentation. For each $j$, let $i$ be minimal such that
$S_j\subseteq K\cap M_i$. Then the induced map $S_j\to P_i$ is nonzero and hence identifies $S_j$ with the finite-dimensional irreducible subrepresentation of $P_i$; comparison of highest weights gives $i = m + 1 - j$. Consequently,
\[
\frac{K\cap M_i}{K\cap M_{i-1}}
\cong
\begin{cases}
0,
& 1\leq i<m+1-r,\\[2mm]
S_{m+1-i},
& m+1-r\leq i\leq m+1.
\end{cases}
\]

Therefore, setting
\[
\mathcal{V}
=
D_k\otimes\operatorname{Sym}^m(\mathbb{C}^2),
\qquad
\mathcal{N}_i
=
\frac{M_i+K}{K},
\]
we obtain a $G_2$-stable filtration
\[
0=\mathcal{N}_0\subseteq\mathcal{N}_1\subseteq\cdots
\subseteq\mathcal{N}_{m+1}=\mathcal{V}.
\]
Its successive quotients
\[
\mathcal{Q}_i=\mathcal{N}_i/\mathcal{N}_{i-1}
\]
are given as follows:
\begin{enumerate}
    \item For
    \[
    1\leq i<m+1-\min(k-2,m),
    \]
    the subquotients are principal series:
    \[
    \mathcal{Q}_i
    \cong
    \chi_{\varepsilon_{k+m+1-i},
    -\frac{k-1}{2}+m+1-i}
    \times
    \chi_{\varepsilon_{i-1},
    \frac{k-1}{2}+i-1}.
    \]
    If $N_i\neq 0$, then $\mathcal{Q}_i$ is reducible and has the
    unique irreducible generic Jordan--H\"older subquotient
    \[
    D_{|N_i|+1,\frac{m}{2}}.
    \]
    More explicitly, this representation is
    \[
    D_{|N_i|+1,\frac{m}{2}}
    \cong
    \begin{cases}
    D_{m+4-k-2i,\frac{m}{2}},
    & N_i<0,\\[1mm]
    D_{k+2i-m-2,\frac{m}{2}},
    & N_i>0.
    \end{cases}
    \]
    If
    \[
    N_i=0,
    \qquad\text{equivalently}\qquad
    i=\frac{m+3-k}{2},
    \]
    then $\mathcal{Q}_i$ is irreducible and generic, but is not in
    general position.

    \item For
    \[
    m+1-\min(k-2,m)\leq i\leq m+1,
    \]
    the subquotients are twisted discrete series representations:
    \[
    \mathcal{Q}_i
    \cong
    D_{k+2i-m-2,\frac{m}{2}}.
    \]
\end{enumerate}

\begin{lemma}\label{lem:discrete-algebraic-quotient}
Let $k\geq2$ and $a\geq0$ be integers, and let $t\in\mathbb C$. For every integer $p$
satisfying
\[
p\geq2,
\qquad
|p-k|\leq a,
\qquad
p\equiv k+a\pmod2,
\]
there exists a surjective $G_2$-equivariant map
\[
D_{k,t}\otimes\Sym^a(\mathbb C^2)
\longtwoheadrightarrow
D_{p,t+\frac a2}.
\]
\end{lemma}

\begin{proof}
Let
\[
\mathcal{V} = D_{k,t} \otimes \mathrm{Sym}^a(\mathbb{C}^2),
\]
and twist the filtration from Subsection~\ref{discrete-twist} by $|\det|^t$. Write the resulting filtration as
\[
0 = \mathcal{N}_0 \subseteq \mathcal{N}_1 \subseteq \cdots \subseteq \mathcal{N}_{a+1} = \mathcal{V},
\]
with successive quotients $\mathcal{Q}_i = \mathcal{N}_i / \mathcal{N}_{i-1}$.

Define the index
\[
i_+ = \frac{a+p-k+2}{2}.
\]
The congruence condition $p \equiv k+a \pmod 2$ ensures that $i_+$ is an integer, while the bound $|p-k| \leq a$ guarantees that $1 \leq i_+ \leq a+1$. 

Recall that $N_i = k+2i-a-3$. Substituting $i_+$, we find that $N_{i_+} = p-1 > 0$. The description of the graded quotients therefore gives a surjection
\[
\mathcal{Q}_{i_+} \longtwoheadrightarrow D_{p, t+\frac{a}{2}}.
\]

Let $\zeta$ denote the infinitesimal character of $D_{p, t+a/2}$, which is represented by the unordered pair
\[
\left\{ t+\frac{a-p+1}{2}, t+\frac{a+p-1}{2} \right\}.
\]
In contrast, the infinitesimal character of each quotient $\mathcal{Q}_i$ is represented by
\[
\left\{ t+\frac{a-N_i}{2}, t+\frac{a+N_i}{2} \right\}.
\]
Comparing these pairs shows that the generalized infinitesimal-character summand $\mathcal{Q}_i[\zeta]$ is nonzero if and only if $N_i = \pm(p-1)$. Because $N_i$ strictly increases with respect to $i$, the index $i_+$ (where $N_{i_+} = p-1$) is the largest index satisfying this condition.

Now consider the quotient representation $Y = \mathcal{V} / \mathcal{N}_{i_+}$. This representation admits a finite filtration whose successive quotients are $\mathcal{Q}_{i_+ + 1}, \ldots, \mathcal{Q}_{a+1}$. As established above, none of these quotients has infinitesimal character $\zeta$, so $Y[\zeta] = 0$. By the exactness of generalized infinitesimal-character summands for finite-length Casselman–Wallach representations, recorded in Lemma \ref{lem:primary-decomposition}, the short exact sequence
\[
0 \longrightarrow \mathcal{N}_{i_+} \longrightarrow \mathcal{V} \longrightarrow Y \longrightarrow 0
\]
therefore gives $\mathcal{V}[\zeta] = \mathcal{N}_{i_+}[\zeta]$.

Applying this same exact functor to the sequence
\[
0 \longrightarrow \mathcal{N}_{i_+-1} \longrightarrow \mathcal{N}_{i_+} \longrightarrow \mathcal{Q}_{i_+} \longrightarrow 0
\]
yields the isomorphism
\[
\frac{\mathcal{N}_{i_+}[\zeta]}{\mathcal{N}_{i_+-1}[\zeta]} \cong \mathcal{Q}_{i_+}[\zeta].
\]
Since $\mathcal{Q}_{i_+}$ already has the infinitesimal character $\zeta$, it follows that $\mathcal{Q}_{i_+}[\zeta] = \mathcal{Q}_{i_+}$. Substituting $\mathcal{V}[\zeta] = \mathcal{N}_{i_+}[\zeta]$ into the numerator yields
\[
\frac{\mathcal{V}[\zeta]}{\mathcal{N}_{i_+-1}[\zeta]} \cong \mathcal{Q}_{i_+}.
\]

Finally, the primary decomposition for the action of $Z(\mathfrak{g}_{2,\mathbb C})$  provides a $G_2$-equivariant projection $\mathcal{V} \longtwoheadrightarrow \mathcal{V}[\zeta]$. By composing this projection with the natural quotient map and our previously established surjection, we construct the required sequence of surjections:
\[
\mathcal{V} \longtwoheadrightarrow \mathcal{V}[\zeta] \longtwoheadrightarrow \mathcal{Q}_{i_+} \longtwoheadrightarrow D_{p, t+\frac{a}{2}}. \qedhere
\]
\end{proof}

\subsection{Algebraic Twist of Irreducible Generic Representations}
\label{generic-twist}

Let
\[
\pi_2=
\left(\bigtimes_{i=1}^{r}\chi_{\epsilon_i,s_i}\right)
\times
\left(\bigtimes_{j=1}^{l}D_{k_j,t_j}\right),
\qquad n=r+2l,
\]
be an irreducible generic representation of $G_n$,
where $\epsilon_i\in\{0,1\}$, $k_j\geq 2$, and
$s_i,t_j\in\mathbb C$.  Let
\[
P=P_{1,\ldots,1,2,\ldots,2}=MU
\]
be the standard parabolic subgroup corresponding to the partition
$n=r+2l$, so that
\[
M\cong G_1^r\times G_2^l.
\]
Write
\[
\sigma=
\left(\bigboxtimes_{i=1}^{r}\chi_{\epsilon_i,s_i}\right)
\boxtimes
\left(\bigboxtimes_{j=1}^{l}D_{k_j,t_j}\right),
\qquad
\pi_2=\Ind_P^{G_n}(\sigma).
\]

Let $V=\mathbb C^n$ and consider the flag
\[
0=V_0\subseteq V_1\subseteq\cdots\subseteq V_{r+l}=V
\]
stabilized by $P$, where
\[
\dim(V_j/V_{j-1})=
\begin{cases}
1,&1\leq j\leq r,\\
2,&r+1\leq j\leq r+l.
\end{cases}
\]
Put
\[
W_j=V_j/V_{j-1},
\qquad 1\leq j\leq r+l.
\]

We now construct a finite $P$-stable filtration of
$\rho=\Sym^m(V)$. Set
\[
c_j=(m+2)^{j-1},
\qquad 1\leq j\leq r+l,
\]
and equip $V$ with the increasing filtration for which the graded
piece $W_j$ has degree $c_j$. Equip $\Sym^m(V)$ with the induced
multiplicative filtration. Thus, $\mathcal F_d\rho$ is spanned by
monomials $v_1\cdots v_m$ for which there exist indices
$j_1,\ldots,j_m$ such that
\[
v_q\in V_{j_q},
\qquad
c_{j_1}+\cdots+c_{j_m}\leq d.
\]
Each $\mathcal F_d\rho$ is $P$-stable.

Let
\[
\mathcal A_m=
\left\{
\mathbf a=(a_1,\ldots,a_{r+l})\in
\mathbb Z_{\geq0}^{r+l}
\,|\,
a_1+\cdots+a_{r+l}=m
\right\},
\]
and define
\[
w(\mathbf a)=\sum_{j=1}^{r+l}a_jc_j.
\]
Since $0\leq a_j\leq m$, the map
$\mathbf a\mapsto w(\mathbf a)$ is injective. Order the elements of
$\mathcal A_m$ as
\[
\mathbf a^{(1)},\ldots,\mathbf a^{(N)},
\qquad
w(\mathbf a^{(1)})<\cdots<w(\mathbf a^{(N)}),
\]
where
\[
N=\binom{m+r+l-1}{m},
\]
and define
\[
F^0\rho=0,
\qquad
F^i\rho=\mathcal F_{w(\mathbf a^{(i)})}\rho,
\qquad 1\leq i\leq N.
\]
Then
\[
0=F^0\rho\subseteq F^1\rho\subseteq\cdots\subseteq F^N\rho=\rho
\]
is a $P$-stable filtration, and its successive quotients are
\[
F^i\rho/F^{i-1}\rho
\cong
\bigotimes_{u=1}^{r+l}
\Sym^{a_u^{(i)}}(W_u)
\]
as $M$-representations, where the unipotent radical $U$ acts trivially
on the right-hand side.

The tensor identity gives
\[
\pi_2\widehat{\otimes}\rho
\cong
\Ind_P^{G_n}\bigl(\sigma\otimes\rho|_P\bigr).
\]
Tensoring the above $P$-stable filtration of $\rho$ with $\sigma$ and
using exactness of parabolic induction yields a $G_n$-stable
filtration
\[
0=M_0\subseteq    M_1\subseteq\cdots\subseteq M_N
=\pi_2\widehat{\otimes}\rho,
\]
where
\[
M_i=
\Ind_P^{G_n}\bigl(\sigma\otimes F^i\rho\bigr).
\]
The graded pieces are
\[
\begin{aligned}
\sigma_i
\coloneqq M_i/M_{i-1}
&\cong
\Ind_P^{G_n}
\left(
\sigma\otimes
\left(F^i\rho/F^{i-1}\rho\right)
\right)\\
&\cong
\Ind_P^{G_n}
\Bigg(
\bigboxtimes_{u=1}^{r}
\left(
\chi_{\epsilon_u,s_u}
\otimes
\Sym^{a_u^{(i)}}(\mathbb C)
\right)
\boxtimes
\bigboxtimes_{v=1}^{l}
\left(
D_{k_v,t_v}
\otimes
\Sym^{a_{r+v}^{(i)}}(\mathbb C^2)
\right)
\Bigg).
\end{aligned}
\]
Here
\[
\chi_{\epsilon_u,s_u}\otimes\Sym^{a_u^{(i)}}(\mathbb C)
\cong
\chi_{\epsilon_u+a_u^{(i)},\,s_u+a_u^{(i)}},
\]
where the parity parameter is read modulo $2$.

By Subsection~\ref{discrete-twist}, for each
$v\in\{1,\ldots,l\}$, the representation
\[
D_{k_v,t_v}
\otimes
\Sym^{a_{r+v}^{(i)}}(\mathbb C^2)
\]
admits a $G_2$-stable filtration
\[
0=E_{v,0}^{(i)}
\subseteq E_{v,1}^{(i)}
\subseteq\cdots\subseteq
E_{v,a_{r+v}^{(i)}+1}^{(i)}
=
D_{k_v,t_v}
\otimes
\Sym^{a_{r+v}^{(i)}}(\mathbb C^2).
\]
The successive quotients in this filtration are the principal series
and twisted discrete series representations described in
Subsection~\ref{discrete-twist}; in particular, some of them may be
reducible principal series.

Taking tensor products of these filtrations for
$v=1,\ldots,l$ yields a filtration of
\[
\bigboxtimes_{v=1}^{l}
\left(
D_{k_v,t_v}
\otimes
\Sym^{a_{r+v}^{(i)}}(\mathbb C^2)
\right).
\]
Tensoring this filtration with
\[
\bigboxtimes_{u=1}^{r}
\left(
\chi_{\epsilon_u,s_u}
\otimes
\Sym^{a_u^{(i)}}(\mathbb C)
\right)
\]
and using exactness of parabolic induction, we obtain a filtration of
$\sigma_i$. More precisely, there exists a filtration
\[
0=\sigma_{i,0}
\subseteq\sigma_{i,1}
\subseteq\cdots\subseteq
\sigma_{i,T_i}
=\sigma_i.
\]
Every successive quotient of this filtration is of the form
\[
\Ind_P^{G_n}
\left(
\bigboxtimes_{u=1}^{r}
\left(
\chi_{\epsilon_u,s_u}
\otimes
\Sym^{a_u^{(i)}}(\mathbb C)
\right)
\boxtimes
\bigboxtimes_{v=1}^{l}\tau_v
\right),
\]
where $\tau_v$ is a successive quotient of the filtration of
\[
D_{k_v,t_v}
\otimes
\Sym^{a_{r+v}^{(i)}}(\mathbb C^2).
\]
Thus, these successive quotients need not be irreducible.

Taking the inverse images of the subrepresentations
$\sigma_{i,j}\subseteq\sigma_i=M_i/M_{i-1}$ under the quotient map
\[
M_i\longrightarrow M_i/M_{i-1}
\]
and splicing the resulting filtrations together for
$i=1,\ldots,N$ yields a $G_n$-stable filtration
\[
0=L_0\subseteq L_1\subseteq\cdots\subseteq L_T
=\pi_2\widehat{\otimes}\rho
\]
which refines the filtration
\[
0=M_0\subseteq M_1\subseteq\cdots\subseteq M_N
=\pi_2\widehat{\otimes}\rho.
\]

Now let $\pi_1\in\operatorname{Irr}_{\mathrm{gen}}(G_n)$. Assume that $s=s_0$ is an exceptional pole of type~\(2\) with
level $m\geq0$ for the pair $(\pi_1,\pi_2)$. Thus
\[
\operatorname{Hom}_{G_n}
\left(
\pi_1\widehat{\otimes}\pi_2\widehat{\otimes}\rho,
|\det|^{-s_0}
\right)
\neq0.
\]

\begin{lemma}\label{lem:graded-piece-map}
There exists an integer $1\leq i\leq T$ such that
\[
\operatorname{Hom}_{G_n}
\left(
L_i/L_{i-1},
\pi_1^\vee\otimes|\det|^{-s_0}
\right)
\neq0.
\]
\end{lemma}

\begin{proof}
By assumption, there exists
\[
0\neq\ell\in
\operatorname{Hom}_{G_n}
\left(
\pi_1\widehat{\otimes}L_T,
|\det|^{-s_0}
\right).
\]
By the tensor--Hom adjunction for Casselman--Wallach representations,
$\ell$ determines a nonzero continuous $G_n$-equivariant map
\[
\Phi:L_T\longrightarrow
\pi_1^\vee\otimes|\det|^{-s_0}.
\]
Since the filtration is finite, let $i\geq1$ be minimal such that
$\Phi|_{L_i}\neq0$. Then
\[
L_{i-1}\subseteq\ker(\Phi),
\]
so $\Phi|_{L_i}$ descends to a nonzero continuous
$G_n$-equivariant map
\[
\overline\Phi:
L_i/L_{i-1}\longrightarrow
\pi_1^\vee\otimes|\det|^{-s_0}. \qedhere
\]
\end{proof}

\begin{lemma}\label{lem:rs-parameter-data}
Let $\pi_1,\pi_2\in\operatorname{Irr}_{\mathrm{gen}}(G_n)$, and
suppose that $s=s_0$ is an exceptional pole of type~$2$ with level
$m\geq0$ for $(\pi_1,\pi_2)$. For $i=1,2$, let
$\lambda_i=(\lambda_{i,1},\ldots,\lambda_{i,n})$ be the
infinitesimal-character parameter obtained from the inducing blocks
of $\pi_i$ by replacing a character $\chi_{\epsilon,x}$ by $x$ and
a discrete-series block $D_{k,t}$ by
\[
t-\frac{k-1}{2},\qquad t+\frac{k-1}{2}.
\]
Then there exist $w\in S_n$ and integers $b_1,\ldots,b_n\geq0$,
with $\sum_{\alpha=1}^{n}b_\alpha=m$, such that
\begin{equation}\label{eq:rs-parameter-matching}
s_0+\lambda_{1,w(\alpha)}+\lambda_{2,\alpha}=-b_\alpha,
\qquad 1\leq\alpha\leq n.
\end{equation}
Moreover, if both representations are principal series, written as
\[
\pi_i=\bigtimes_{\alpha=1}^{n}
\chi_{\epsilon_{i,\alpha},\lambda_{i,\alpha}},
\qquad i=1,2,
\]
then $w$ and $(b_1,\ldots,b_n)$ can be chosen so that
\begin{equation}\label{eq:rs-matching-parity}
\epsilon_{1,w(\alpha)}+\epsilon_{2,\alpha}
\equiv b_\alpha\pmod2
\end{equation}
for at least one index $\alpha$.
\end{lemma}

\begin{proof}
Put $\tau=\pi_1^\vee\otimes|\det|^{-s_0}$. By
Lemma~\ref{lem:graded-piece-map}, a graded piece $Q=L_i/L_{i-1}$
admits a nonzero map to $\tau$. Let
$\mathbf a=(a_1,\ldots,a_{r+l})$ be the multi-index of the outer
graded piece from which $Q$ arises. A character block contributes
$s_u+a_u$ to the infinitesimal character of $Q$. For a
discrete-series block, the $i_v$-th inner graded quotient contributes
\[
t_v-\frac{k_v-1}{2}+(a_{r+v}+1-i_v),\qquad
t_v+\frac{k_v-1}{2}+(i_v-1),
\]
where $1\leq i_v\leq a_{r+v}+1$. Thus the infinitesimal character
of $Q$ is represented by
\[
\{\lambda_{2,\alpha}+b_\alpha\,|\,1\leq\alpha\leq n\},
\qquad b_\alpha\geq0,\quad\sum_{\alpha=1}^{n}b_\alpha=m.
\]
Since $Q$ has an infinitesimal character and maps nontrivially to
$\tau$, this multiset equals
$\{-s_0-\lambda_{1,\alpha}:1\leq\alpha\leq n\}$.
Matching occurrences, including repeated entries, proves
\eqref{eq:rs-parameter-matching}.

Suppose now that both representations are principal series. In this
case the selected graded piece has the form
\[
Q=\bigtimes_{u=1}^{n}
\chi_{\epsilon_{2,u}+a_u,\,\lambda_{2,u}+a_u},
\qquad a_u\geq0,\quad\sum_{u=1}^{n}a_u=m.
\]
If $Q$ is irreducible, then $Q\cong\tau$. By uniqueness of the
inducing data for irreducible generic principal series, the inducing
characters agree up to reordering. Matching these characters gives
\eqref{eq:rs-parameter-matching} with $b_u=a_u$ and
\eqref{eq:rs-matching-parity} for every index.

If $Q$ is reducible, the Speh--Vogan criterion \cite{SV1980}
gives distinct indices $u,v$ such that
\[
\Delta'=(\lambda_{2,u}+a_u)-(\lambda_{2,v}+a_v)
\in\mathbb Z\setminus\{0\},
\qquad
\Delta'\equiv
\epsilon_{2,u}+a_u+\epsilon_{2,v}+a_v-1\pmod2.
\]
Writing $\Delta=\lambda_{2,u}-\lambda_{2,v}$ and cancelling the
integral shifts gives
\[
\Delta\in\mathbb Z,
\qquad
\Delta\equiv\epsilon_{2,u}+\epsilon_{2,v}-1\pmod2.
\]
If $\Delta\neq0$, these conditions would contradict the
irreducibility criterion for the original principal series $\pi_2$,
with the two indices ordered appropriately. Hence
\[
\lambda_{2,u}=\lambda_{2,v},
\qquad
\epsilon_{2,u}\not\equiv\epsilon_{2,v}\pmod2.
\]

Choose a matching in \eqref{eq:rs-parameter-matching} with
$b_\alpha=a_\alpha$. If \eqref{eq:rs-matching-parity} holds at
$u$, there is nothing to prove. Otherwise, simultaneously
interchange $b_u,b_v$ and $w(u),w(v)$, leaving all other entries
unchanged. Since $\lambda_{2,u}=\lambda_{2,v}$, the new choices
still satisfy \eqref{eq:rs-parameter-matching}. Moreover, the
nonnegativity of the $b_\alpha$ and the equality
$\sum_{\alpha=1}^{n}b_\alpha=m$ are preserved. The opposite
parities of $\epsilon_{2,u}$ and $\epsilon_{2,v}$ imply that
\eqref{eq:rs-matching-parity} now holds at $v$. This completes
the proof.
\end{proof}

For the following argument and Proposition~\ref{principal-cong-inverse},
assume that both $\pi_1$ and $\pi_2$ are in general position. Choose $i$ as in Lemma~\ref{lem:graded-piece-map}, and put
\[
Q=L_i/L_{i-1},
\qquad
\tau=\pi_1^\vee\otimes|\det|^{-s_0},
\]
and let $\mathbf a=(a_1,\ldots,a_{r+l})$ be the multi-index of the
outer graded piece from which $Q$ arises. By construction, there are
successive quotients $\tau_v$ of the filtrations of
\[
D_{k_v,t_v}\otimes\Sym^{a_{r+v}}(\mathbb C^2),
\qquad 1\leq v\leq l,
\]
such that
\[
Q\cong
\operatorname{Ind}_{P}^{G_n}
\left(
\bigboxtimes_{u=1}^{r}
\chi_{\epsilon_u+a_u,\,s_u+a_u}
\boxtimes
\bigboxtimes_{v=1}^{l}\tau_v
\right).
\]
Each $\tau_v$ is either a twisted discrete series representation or a
principal series representation occurring in
Subsection~\ref{discrete-twist}.

By induction in stages, $Q$ is parabolically induced from essentially square-integrable representations and thus has a unique irreducible generic Jordan--Hölder subquotient, denoted by $Q_0$. By Lemma~\ref{lem:graded-piece-map}, there exists a nonzero $G_n$-equivariant map
\[
f: Q \longrightarrow \tau.
\]
Since $\tau$ is irreducible and generic, it follows that $\tau \cong Q_0$.

For each index $v$, let $i_v$ be the position of the graded quotient $\tau_v$ in the filtration of
\[
D_{k_v, t_v} \otimes \mathrm{Sym}^{a_{r+v}}(\mathbb{C}^2),
\]
and set
\[
N_{v,i_v} = k_v + 2i_v - a_{r+v} - 3.
\]
Let $\gamma_v$ denote the unique irreducible generic Jordan--Hölder constituent of $\tau_v$. By Subsection~\ref{discrete-twist},
\[
\gamma_v \cong
\begin{cases}
D_{|N_{v,i_v}|+1, \, t_v + \frac{a_{r+v}}{2}}, & N_{v,i_v} \neq 0, \\[2mm]
\tau_v, & N_{v,i_v} = 0.
\end{cases}
\]

Define the representation
\[
\Sigma_{\mathbf{a}}
=
\left(
\bigtimes_{u=1}^{r}
\chi_{\epsilon_u+a_u, \, s_u+a_u}
\right)
\times
\left(
\bigtimes_{v=1}^{l} \gamma_v
\right).
\]
By the exactness of normalized parabolic induction, $\Sigma_{\mathbf a}$ occurs as a subquotient of $Q$.

We claim that \(\Sigma_{\mathbf a}\) is irreducible and generic.
For blocks arising from distinct original inducing blocks of \(\pi_2\),
the character--character exponent differences remain outside
\(\mathbb Z\), while the character--discrete and discrete--discrete
exponent differences remain outside \(\tfrac12\mathbb Z\), since the
shifts introduced above are integral or half-integral.

If \(N_{v,i_v}=0\), then
\[
\gamma_v=\tau_v
\cong
\chi_{\delta_v,\,t_v+\frac{a_{r+v}}2}
\times
\chi_{\delta_v+1,\,t_v+\frac{a_{r+v}}2}
\]
for some \(\delta_v\in\{0,1\}\). This rank-one principal series is
irreducible: its two exponents are equal and its parity parameters are
opposite.

It follows from the Speh--Vogan reducibility
criterion~\cite{SV1980} that \(\Sigma_{\mathbf a}\) is irreducible.
Its genericity follows from the heredity of Whittaker models.

Thus, $\Sigma_{\mathbf a}$ is an irreducible generic Jordan--Hölder constituent of $Q$. The uniqueness of $Q_0$ gives
\[
\Sigma_{\mathbf{a}} \cong Q_0.
\]
Combining this with $Q_0 \cong \tau$, we obtain
\[
\Sigma_{\mathbf{a}}
\cong
\pi_1^\vee \otimes |\det|^{-s_0}.
\]

We now show that $N_{v,i_v}\neq0$ for every $v$. Suppose that
$N_{v,i_v}=0$ for some $v$. Since $\Sigma_{\mathbf a}$ and $\tau$
are isomorphic irreducible generic inductions, uniqueness of their
essentially square-integrable inducing data, up to reordering,
implies that the two character blocks in $\gamma_v$ also occur
among the inducing blocks of
\[
\tau=\pi_1^\vee\otimes|\det|^{-s_0}.
\]
Taking contragredients and applying a common determinant twist preserve equality of character exponents. Their difference is \(0\in\mathbb Z\),
which contradicts the general-position condition
\(s_i-s_{i'}\notin\mathbb Z\) for the character blocks of \(\pi_1\).
It follows that
\[
\gamma_v \cong
D_{p_v, \, t_v + \frac{a_{r+v}}{2}},
\]
where
\[
p_v \geq 2,
\qquad
|p_v - k_v| \leq a_{r+v},
\qquad
p_v \equiv k_v + a_{r+v} \pmod 2.
\]
Consequently,
\[
\Sigma_{\mathbf{a}}
=
\left(
\bigtimes_{u=1}^{r}
\chi_{\epsilon_u+a_u, \, s_u+a_u}
\right)
\times
\left(
\bigtimes_{v=1}^{l}
D_{p_v, \, t_v + \frac{a_{r+v}}{2}}
\right).
\]

\begin{prop}\label{principal-cong-inverse}
Let
\[
\pi_2=
\left(
\bigtimes_{u=1}^{r}\chi_{\epsilon_u,s_u}
\right)
\times
\left(
\bigtimes_{v=1}^{l}D_{k_v,t_v}
\right),
\qquad n=r+2l,
\]
and let $\pi_1\in\operatorname{Irr}_{\mathrm{gen}}(G_n)$. Assume that
both $\pi_1$ and $\pi_2$ are in general position. Let $s_0\in\mathbb C$ and $m\in\mathbb Z_{\geq0}$. Then $s=s_0$ is an exceptional pole of
type~\(2\) with level $m$ for the pair $(\pi_1,\pi_2)$ if and only if
there exist integers
\[
a_1,\ldots,a_{r+l}\geq0,
\qquad
\sum_{u=1}^{r+l}a_u=m,
\]
and integers $p_v$, $1\leq v\leq l$, satisfying
\[
p_v\geq2,
\qquad
|p_v-k_v|\leq a_{r+v},
\qquad
p_v\equiv k_v+a_{r+v}\pmod2,
\]
such that, after reordering the inducing blocks,
\[
\pi_1\cong
\left(
\bigtimes_{u=1}^{r}
\chi_{\epsilon_u+a_u,\,-(s_u+a_u+s_0)}
\right)
\times
\left(
\bigtimes_{v=1}^{l}
D_{p_v,\,-\left(t_v+\frac{a_{r+v}}2+s_0\right)}
\right),
\]
where the parity parameters are read modulo $2$.
\end{prop}

\begin{proof}
The forward implication was proved above. For the converse, choose
integers $a_1,\ldots,a_{r+l}$ and $p_1,\ldots,p_l$ satisfying the
stated conditions. Put
\[
\Sigma=
\left(
\bigtimes_{u=1}^{r}
\chi_{\epsilon_u+a_u,\,s_u+a_u}
\right)
\times
\left(
\bigtimes_{v=1}^{l}
D_{p_v,\,t_v+\frac{a_{r+v}}2}
\right).
\]
Then $\Sigma \cong \pi_1^\vee \otimes |\det|^{-s_0}$.

Let $\mathbf{a} = (a_1, \ldots, a_{r+l})$. The outer filtration constructed above has a graded quotient
\[
\begin{aligned}
\sigma_{\mathbf a}
\cong
\operatorname{Ind}_{P}^{G_n}\Bigg(
&
\bigboxtimes_{u=1}^{r}
\chi_{\epsilon_u+a_u,\,s_u+a_u}
\boxtimes
\bigboxtimes_{v=1}^{l}
\left(
D_{k_v,t_v} \otimes
\Sym^{a_{r+v}}(\mathbb C^2)
\right)
\Bigg).
\end{aligned}
\]
By Lemma~\ref{lem:discrete-algebraic-quotient}, for every $v$ there is a surjection
\[
D_{k_v,t_v} \otimes \Sym^{a_{r+v}}(\mathbb C^2) \longtwoheadrightarrow D_{p_v,\,t_v+\frac{a_{r+v}}2}.
\]
Taking external tensor products and using exactness of normalized parabolic induction, we obtain a surjection $\sigma_{\mathbf a} \longtwoheadrightarrow \Sigma$.

Let $\mathcal V = \pi_2 \widehat{\otimes} \Sym^m(\mathbb C^n)$, and write $\sigma_{\mathbf b}$ for the outer graded quotient corresponding to $\mathbf b = (b_1, \ldots, b_{r+l}) \in \mathcal A_m$. Let $j$ be the index for which $M_j/M_{j-1} \simeq \sigma_{\mathbf a}$, and let $\zeta$ be the infinitesimal character of $\Sigma$.

We claim that
\[
\sigma_{\mathbf b}[\zeta] = 0
\qquad
\text{whenever } \mathbf b \neq \mathbf a.
\]
Indeed, every infinitesimal character occurring in $\sigma_{\mathbf b}$ contains, for each $u$, the entry $s_u+b_u$. For each $v$, the two entries arising from the $v$-th $G_2$-block belong to $t_v+\frac12\mathbb Z$ and have sum $2t_v+b_{r+v}$. On the other hand, the infinitesimal character of $\Sigma$ contains the entry $s_u+a_u$ from the $u$-th $G_1$-block, while the two entries from the $v$-th $G_2$-block have sum $2t_v+a_{r+v}$.

By the general-position assumption on $\pi_2$, the cosets $s_u+\mathbb Z$ and $t_v+\frac12\mathbb Z$ attached to the different original inducing blocks are pairwise disjoint. Hence equality with the infinitesimal character of $\Sigma$ cannot mix entries arising from different original blocks. It therefore implies $b_u=a_u$ for $1 \leq u \leq r$, and $b_{r+v}=a_{r+v}$ for $1 \leq v \leq l$. Thus $\mathbf b=\mathbf a$, proving the claim.

 By Lemma~\ref{lem:primary-decomposition} and exactness along the outer filtration, the claim gives
\[
M_{j-1}[\zeta] = 0
\qquad\text{and}\qquad
(\mathcal V/M_j)[\zeta] = 0.
\]
Applying $X \mapsto X[\zeta]$ to
\[
0 \longrightarrow M_{j-1} \longrightarrow M_j \longrightarrow \sigma_{\mathbf a} \longrightarrow 0
\]
and
\[
0 \longrightarrow M_j \longrightarrow \mathcal V \longrightarrow \mathcal V/M_j \longrightarrow 0,
\]
we obtain $\mathcal V[\zeta] \simeq \sigma_{\mathbf a}[\zeta]$.

Since $\Sigma$ has infinitesimal character $\zeta$, the surjection $\sigma_{\mathbf a} \longtwoheadrightarrow \Sigma$ factors through a surjection $\sigma_{\mathbf a}[\zeta] \longtwoheadrightarrow \Sigma$. Composing with the primary projection gives
\[
\mathcal V
\longtwoheadrightarrow
\mathcal V[\zeta]
\cong
\sigma_{\mathbf a}[\zeta]
\longtwoheadrightarrow
\Sigma
\cong
\pi_1^\vee \otimes |\det|^{-s_0}.
\]
By tensor--Hom adjunction,
\[
\operatorname{Hom}_{G_n}
\left(
\pi_1 \widehat{\otimes} \pi_2 \widehat{\otimes} \Sym^m(\mathbb C^n),
|\det|^{-s_0}
\right)
\neq 0.
\]
Therefore $s=s_0$ is an exceptional pole of type $2$ with level $m$ for the pair $(\pi_1,\pi_2)$.
\end{proof}

We now return to arbitrary irreducible generic pairs and show that
every exceptional pole of type~$2$ is a pole of the Rankin--Selberg
$L$-function. In general position, we also compute its exact order.

\begin{prop}\label{pole}
Let
\[
\pi_2=
\left(\bigtimes_{i=1}^{r}\chi_{\epsilon_i,s_i}\right)
\times
\left(\bigtimes_{j=1}^{l}D_{k_j,t_j}\right),
\qquad n=r+2l,
\]
be irreducible and generic, and let
$\pi_1\in\operatorname{Irr}_{\mathrm{gen}}(G_n)$.
Suppose that $s=s_0$ is an exceptional pole of type~$2$ with
level $m\geq0$ for $(\pi_1,\pi_2)$. Then
$L(s,\pi_1\times\pi_2)$ has a pole at $s=s_0$ of order
$d\geq1$.

If, in addition, both representations are in general position, let
$\mathbf a$ and $p_1,\ldots,p_l$ be the parameters in
Proposition~\ref{principal-cong-inverse}. Then
\[
d=r+l+
\#\left\{1\leq j\leq l\,|\,
a_{r+j}\geq p_j+k_j-2\right\}.
\]
\end{prop}

\begin{proof}
By multiplicativity, $L(s,\pi_1\times\pi_2)$ is the product of the
Rankin--Selberg factors attached to pairs of inducing blocks.
The local Gamma-factor formulas \cite{Kna1994} are
\[
\begin{aligned}
L(s,\chi_{\epsilon,x}\times\chi_{\eta,y})
&=
\Gamma_{\mathbb R}
\left(s+x+y+\varepsilon_{\epsilon+\eta}\right),\\
L(s,\chi_{\epsilon,x}\times D_{k,t})
&=
\Gamma_{\mathbb C}
\left(s+x+t+\frac{k-1}{2}\right),\\
L(s,D_{p,u}\times D_{k,t})
&=
\Gamma_{\mathbb C}
\left(s+u+t+\frac{p+k-2}{2}\right)
\Gamma_{\mathbb C}
\left(s+u+t+\frac{|p-k|}{2}\right),
\end{aligned}
\]
and these factors are symmetric in their two arguments.
Recall that $\Gamma_{\mathbb R}$ and $\Gamma_{\mathbb C}$ have
simple poles at $2\mathbb Z_{\leq0}$ and $\mathbb Z_{\leq0}$,
respectively, and neither has zeros. Thus a pole of any block factor
gives a pole of $L(s,\pi_1\times\pi_2)$, and the pole orders of the
block factors add.

We apply these formulas to the parameter matching in
Lemma~\ref{lem:rs-parameter-data}. Suppose first that a
discrete-series block occurs in at least one representation.
Choose such a block $D_{k,t}$ with $k$ maximal among all
discrete-series indices occurring in either representation, and
consider the entry $t+(k-1)/2$. If its matched entry is a character
exponent $x$ of the other representation, then
\[
s_0+x+t+\frac{k-1}{2}=-b
\qquad\text{for some }b\in\mathbb Z_{\geq0}.
\]
The corresponding mixed factor therefore has a pole at $s=s_0$.
Otherwise, the matched entry belongs to a block $D_{p,u}$, with
$p\leq k$, and is either $u+(p-1)/2$ or $u-(p-1)/2$. Hence
\[
s_0+u+t+\frac{p+k-2}{2}=-b
\qquad\text{or}\qquad
s_0+u+t+\frac{k-p}{2}=-b
\]
for some $b\in\mathbb Z_{\geq0}$. Since $k-p=|k-p|$, one of
the two Gamma factors in $L(s,D_{p,u}\times D_{k,t})$ has a pole
at $s=s_0$.

If both representations are principal series, choose $w$ and
$(b_1,\ldots,b_n)$ as in Lemma~\ref{lem:rs-parameter-data}, and
let $\alpha$ satisfy \eqref{eq:rs-matching-parity}. By
\eqref{eq:rs-parameter-matching}, the corresponding
character--character factor is
\[
\Gamma_{\mathbb R}
\left(s-s_0-b_\alpha+\varepsilon_{b_\alpha}\right).
\]
Since $-b_\alpha+\varepsilon_{b_\alpha}\in2\mathbb Z_{\leq0}$,
this factor has a simple pole at $s=s_0$. Thus, in all cases,
$L(s,\pi_1\times\pi_2)$ has a pole at $s=s_0$.

When both representations are in general position,
Proposition~\ref{principal-cong-inverse} gives a blockwise
description that allows us to determine all the contributions
to its order. After reordering the inducing blocks, we have
\[
\pi_1\cong
\left(
\bigtimes_{i=1}^{r}
\chi_{\epsilon_i+\varepsilon_{a_i},\,-(s_i+a_i+s_0)}
\right)
\times
\left(
\bigtimes_{j=1}^{l}
D_{p_j,\,-(t_j+a_{r+j}/2+s_0)}
\right),
\]
where $a_1,\ldots,a_{r+l}\geq0$, $\sum_i a_i=m$, and
\[
p_j\geq2,\qquad
|p_j-k_j|\leq a_{r+j},\qquad
p_j\equiv k_j+a_{r+j}\pmod2.
\]

For each $1\leq i\leq r$, the factor attached to the $i$-th
character block of $\pi_1$ and the corresponding block of $\pi_2$ is
\[
\Gamma_{\mathbb R}
\left(s-s_0-a_i+\varepsilon_{a_i}\right).
\]
It has a simple pole at $s=s_0$, since
$-a_i+\varepsilon_{a_i}\in2\mathbb Z_{\leq0}$.
These pairs therefore contribute $r$ to the pole order.

For each $1\leq j\leq l$, the corresponding discrete-series
blocks contribute
\[
\Gamma_{\mathbb C}
\left(s-s_0+\frac{p_j+k_j-2-a_{r+j}}{2}\right)
\Gamma_{\mathbb C}
\left(s-s_0+\frac{|p_j-k_j|-a_{r+j}}{2}\right).
\]
The inequality and parity conditions give
\[
\frac{|p_j-k_j|-a_{r+j}}{2}\in\mathbb Z_{\leq0},
\]
so the second factor always has a simple pole at $s=s_0$.
The argument of the first factor at $s=s_0$ is also an integer,
and it is nonpositive precisely when
$a_{r+j}\geq p_j+k_j-2$. Thus these pairs contribute
\[
l+\#\left\{1\leq j\leq l \,|\,
a_{r+j}\geq p_j+k_j-2\right\}.
\]

It remains to consider the other pairs of blocks. By the same
Gamma-factor formulas, a pole from two noncorresponding character
blocks would imply $s_{i'}-s_i\in\mathbb Z$ for some $i\neq i'$.
A pole from a mixed pair would imply
$s_i-t_j\in\frac12\mathbb Z$, and a pole from two noncorresponding
discrete-series blocks would imply
$t_{j'}-t_j\in\frac12\mathbb Z$ for some $j\neq j'$.
All these possibilities are excluded by general position.
Consequently, only the corresponding block pairs contribute,
and the pole order is
\[
d=r+l+
\#\left\{1\leq j\leq l\,|\,
a_{r+j}\geq p_j+k_j-2\right\}.
\qedhere
\]
\end{proof}
\end{subsection}
\end{section}

\begin{section}{Proof of the Main Theorem}\label{s4}

For $s\in\C$, we define the Hom-space:
\begin{equation}\label{eq:Hspace}
 \mathcal H_s(\pi_1,\pi_2)
 :=\Hom_{G_n}\!\left(
 \pi_1\wh\pi_2\wh\mathcal{S},\detabs^{-s}
 \right).
\end{equation}

\begin{prop}[All-parameter uniqueness]\label{prop:all-s-uniqueness}
Let $\pi_1,\pi_2\in\Irr(G_n)$. For every $s\in\C$, we have
\[
 \dim\mathcal H_s(\pi_1,\pi_2)\le1.
\]
This multiplicity bound also holds if $\detabs^{-s}$ is replaced by any smooth character of $G_n$.
\end{prop}

\begin{proof}
Chen and Sun \cite[Theorem~B]{CS2015} consider the Schwartz space of column vectors
\[
\omega=\mathcal S(\mathbb R^n_{\mathrm{col}}),
\qquad
(\omega(g)\phi)(v)=\phi(g^{-1}v),
\]
and prove that for any irreducible  representations $\tau_1,\tau_2$ of $G_n$ and any character $\chi$,
\begin{equation}\label{eq:CS-theorem-B}
\dim\Hom_{G_n}
(\tau_1\wh\tau_2\wh\omega,\chi)
\leq1.
\end{equation}

Let $\theta(g)={}^tg^{-1}$ and define the continuous linear isomorphism
\[
U:\mathcal S(\mathbb R^n_{\mathrm{row}})
\longrightarrow
\mathcal S(\mathbb R^n_{\mathrm{col}}),
\qquad
(U\Phi)(v)=\Phi({}^tv).
\]
A direct calculation gives
\[
U(\theta(g)\cdot\Phi)=\omega(g)(U\Phi).
\]

For $\pi\in\Irr(G_n)$, put $\pi^\theta(g)=\pi(\theta(g)).$
Since $\theta$ is an automorphism of $G_n$, 
$\pi^\theta\in\Irr(G_n)$. For every smooth character $\chi$, the map $U$
induces an isomorphism
\[
\operatorname{Hom}_{G_n}
\left(
\pi_1\widehat\otimes\pi_2\widehat\otimes\mathcal{S},
\chi
\right)
\cong
\operatorname{Hom}_{G_n}
\left(
\pi_1^\theta\widehat\otimes\pi_2^\theta
\widehat\otimes\omega,
\chi\circ\theta
\right).
\]
The assertion therefore follows from
\cite[Theorem~B]{CS2015}.
\end{proof}

From now on, let $\pi_1,\pi_2\in\operatorname{Irr}_{\mathrm{gen}}(G_n)$. Fix $G_n$-equivariant topological
identifications
\[
\pi_1\cong\mathcal W(\pi_1,\psi),
\qquad
\pi_2\cong\mathcal W(\pi_2,\psi^{-1}),
\]
which we use without further comment. We normalize the local Rankin--Selberg integral:
\begin{equation}\label{eq:Lambda-def}
 \Lambda_s(W,W',\Phi)
 :=\frac{I(s,W,W',\Phi)}{L(s,\pi_1\times\pi_2)}.
\end{equation}
Following Jacquet \cite{Jac2009} and Cogdell--Piatetski-Shapiro \cite{CPS2004}, for any fixed choice of data, this ratio extends to an entire function of $s$. Furthermore, for each fixed $s$, it defines a continuous trilinear functional satisfying the equivariance property:
\begin{equation}\label{eq:Lambda-equivariance}
 \Lambda_s(g\cdot W,g\cdot W',g\cdot\Phi)
 =|\det(g)|^{-s}\Lambda_s(W,W',\Phi).
\end{equation}
Consequently, $\Lambda_s\in\mathcal H_s(\pi_1,\pi_2)$. 
Expanding the algebraic tensors in
\cite[Theorem~2.7(ii)]{Jac2009} into finite sums of pure tensors gives
triples $(W_i,W'_i,\Phi_i)$, indexed by a finite set, such that
\[
\sum_i I(s,W_i,W'_i,\Phi_i)=L(s,\pi_1\times\pi_2).
\]
Dividing by the $L$-factor and using holomorphic continuation gives
\[
\sum_i\Lambda_s(W_i,W'_i,\Phi_i)=1.
\]
Hence $\Lambda_{s_0}\neq0$ for every $s_0\in\mathbb C$.

We realize $\operatorname{Sym}^m(\mathbb C^n)$ as the space of
homogeneous polynomials of degree $m$ in a row variable, with
the action
\[
(g\cdot P)(x)=P(xg).
\]
For $m\geq0$, define
\begin{equation}\label{eq:jet-map}
J_m:\mathcal{S}\longrightarrow
\operatorname{Sym}^m(\mathbb C^n)
\end{equation}
by taking the homogeneous component of degree $m$ in the Taylor
expansion at the origin:
\[
J_m\Phi
=
\sum_{|\alpha|=m}
\frac{D^\alpha\Phi(0)}{\alpha!}X^\alpha.
\]
The chain rule shows that $J_m$ is $G_n$-equivariant. It is
continuous and surjective, and its restriction to the
filtration satisfies
\begin{equation}\label{eq:jet-properties}
J_m(\mathcal{S}^{m+1})=0,
\qquad
J_m|_{\mathcal{S}^m}:
\mathcal{S}^m\longtwoheadrightarrow
\operatorname{Sym}^m(\mathbb C^n),
\qquad
\ker(J_m|_{\mathcal{S}^m})
=
\mathcal{S}^{m+1}.
\end{equation}
Indeed, surjectivity follows by multiplying a homogeneous
polynomial of degree $m$ by a compactly supported smooth
function which is identically $1$ near the origin.

\begin{lemma}[The jet form]\label{lem:jet-form}
Assume that $s=s_0$ is an exceptional pole of type~$2$ with level $m$. Then there exists a form
\[
 T_m\in\mathcal H_{s_0}(\pi_1,\pi_2)
\]
that vanishes identically on $\mathcal{S}^{m+1}$ and is not identically zero on $\mathcal{S}^m$.
\end{lemma}

\begin{proof}
By hypothesis, choose a nonzero functional
\[
 0\ne\ell_m\in
 \Hom_{G_n}\!\left(
 \pi_1\wh\pi_2\wh\Sym^m(\C^n),\detabs^{-s_0}
 \right).
\]
Transferring this to the Whittaker models, we define
\begin{equation}\label{eq:Tm-def}
 T_m(W,W',\Phi):=\ell_m(W,W',J_m\Phi).
\end{equation}
The continuity and $G_n$-equivariance of $T_m$ follow from those
of $\ell_m$ and $J_m$. The first property in \eqref{eq:jet-properties} guarantees that $T_m$ vanishes on $\mathcal{S}^{m+1}$. Because $\ell_m\ne0$ and the restriction $J_m|_{\mathcal{S}^m}$ is surjective, $T_m$ remains nontrivial on $\mathcal{S}^m$.
\end{proof}

\begin{thm}[Type~\(2\) implies Type~\(1\)]
\label{thm:type2-to-type1}
Let $\pi_1,\pi_2\in\operatorname{Irr}_{\mathrm{gen}}(G_n).
$
If $s=s_0$ is an exceptional pole of
type~\(2\) with level $m$, then it is an exceptional pole of
type~\(1\) with level $m$.
\end{thm}

\begin{proof}
We have
\[
 0\ne\Lambda_{s_0}\in\mathcal H_{s_0}(\pi_1,\pi_2).
\]
Because $s=s_0$ is an exceptional pole of type $2$ with level $m$, Lemma \ref{lem:jet-form} gives
\[
 0\ne T_m\in\mathcal H_{s_0}(\pi_1,\pi_2).
\]
Proposition \ref{prop:all-s-uniqueness} states that this Hom-space has dimension at most one. Therefore, there is a scalar $c\in\C^\times$ such that
\begin{equation}\label{eq:Lambda-is-jet}
 \Lambda_{s_0}=cT_m.
\end{equation}
It follows immediately that
\begin{equation}\label{eq:Lambda-level}
 \Lambda_{s_0}|_{\mathcal{S}^{m+1}}=0,
 \qquad
 \Lambda_{s_0}|_{\mathcal{S}^m}\ne0.
\end{equation}

By Proposition~\ref{pole},
$L(s,\pi_1\times\pi_2)$ has a pole at $s=s_0$.
Let $d$ be its order, and write
\[
L(s,\pi_1\times\pi_2)
=
(s-s_0)^{-d}a(s),
\]
where $a$ is holomorphic near $s_0$ and $a(s_0)\neq0$.

For every $W$, $W'$, and $\Phi$, we have
\[
I(s,W,W',\Phi)
=
(s-s_0)^{-d}
a(s)\Lambda_s(W,W',\Phi).
\]
Since $\Lambda_s(W,W',\Phi)$ is holomorphic at $s=s_0$,
every Rankin--Selberg integral has at most a pole of order $d$
at $s=s_0$.

Moreover, $\Lambda_{s_0}$ is a nonzero trilinear functional.
Hence there exist $W_0$, $W'_0$, and $\Phi_0$ such that
\[
\Lambda_{s_0}(W_0,W'_0,\Phi_0)\neq0.
\]
For these data, the integral
$I(s,W_0,W'_0,\Phi_0)$ has a pole of order exactly $d$.
Consequently, $d$ is the maximal order of a pole at $s=s_0$
among the family of Rankin--Selberg integrals.

It follows that the leading Laurent coefficient of the family is
\[
B_{s_0}(W,W',\Phi)
=
a(s_0)\Lambda_{s_0}(W,W',\Phi).
\]
Since $a(s_0)\neq0$, equation~\eqref{eq:Lambda-level} gives
\[
B_{s_0}|_{\mathcal{S}^{m+1}}=0,
\qquad
B_{s_0}|_{\mathcal{S}^m}\neq0.
\]
Therefore $s=s_0$ is an exceptional pole of type~\(1\)
with level $m$.
\end{proof}

We now complete the proof of Theorem \ref{1.1}.

\begin{proof}[Proof of Theorem~\ref{1.1}]
The difficult implication follows from
Theorem~\ref{thm:type2-to-type1}. Conversely, suppose that $s=s_0$ is an exceptional pole of
type~\(1\) with level $m$. The restriction of $B_{s_0}$ to
$\mathcal{S}^m$ vanishes on $\mathcal{S}^{m+1}$ and
therefore descends to a nonzero continuous $G_n$-equivariant
functional on
\[
\pi_1\widehat\otimes\pi_2\widehat\otimes
\left(
\mathcal{S}^m/\mathcal{S}^{m+1}
\right).
\]
Using the $G_n$-equivariant identification
\[
\mathcal{S}^m/\mathcal{S}^{m+1}
\cong
\operatorname{Sym}^m(\mathbb C^n),
\]
we obtain
\[
\operatorname{Hom}_{G_n}\!\left(
\pi_1\widehat\otimes\pi_2\widehat\otimes
\operatorname{Sym}^m(\mathbb C^n),
|\det|^{-s_0}
\right)
\neq0.
\]
Thus $s=s_0$ is exceptional of type~\(2\) with level $m$.
\end{proof}

\end{section}

\begin{section}{Archimedean Derivatives and Exceptional $L$-functions}\label{s5}

In this section, we prove Theorem \ref{1.2}. First, we record the following characterization of exceptional poles.

\begin{thm}\label{classification-general}
Let
\[
\pi_2=
\left(
\bigtimes_{u=1}^{r}\chi_{\epsilon_u,s_u}
\right)
\times
\left(
\bigtimes_{v=1}^{l}D_{k_v,t_v}
\right),
\qquad n=r+2l,
\]
and let
\(\pi_1\in\operatorname{Irr}_{\mathrm{gen}}(G_n)\).
Assume that \(\pi_1\) and \(\pi_2\) are in general position. Let
\(s_0\in\mathbb C\) and \(m\in\mathbb Z_{\geq0}\).
The following conditions are equivalent:
\begin{enumerate}[label=$(\it{\roman*})$]
    \item\label{characterization1}
    $s=s_0$ is an exceptional pole of type~\(2\) with level $m$ for
    $(\pi_1,\pi_2)$.

    \item\label{characterization2}
    There exist integers
    \[
    a_1,\ldots,a_{r+l}\geq0,
    \qquad
    \sum_{u=1}^{r+l}a_u=m,
    \]
    and integers $p_v$ $(1\leq v\leq l)$ satisfying
    \[
    p_v\geq2,
    \qquad
    |p_v-k_v|\leq a_{r+v},
    \qquad
    p_v\equiv k_v+a_{r+v}\pmod2,
    \]
    such that, after reordering the inducing blocks,
    \[
    \pi_1\cong
    \left(
    \bigtimes_{u=1}^{r}
    \chi_{\epsilon_u+a_u,\,-(s_u+a_u+s_0)}
    \right)
    \times
    \left(
    \bigtimes_{v=1}^{l}
    D_{p_v,\,-\left(t_v+\frac{a_{r+v}}2+s_0\right)}
    \right).
    \]

    \item\label{characterization3}
    $s=s_0$ is an exceptional pole of type~\(1\) with level $m$ for
    $(\pi_1,\pi_2)$.
\end{enumerate}
The parity parameters in the characters
\(\chi_{\epsilon,s}\) are read modulo \(2\).
\end{thm}

\begin{proof}
The equivalence of \ref{characterization1} and \ref{characterization2} is
Proposition~\ref{principal-cong-inverse}. The equivalence of \ref{characterization1} and \ref{characterization3} is
Theorem~\ref{1.1}. 
\end{proof}

Theorem~\ref{1.1} allows us to use a unified notion of exceptional poles.

\begin{defn}
Let $\pi_1,\pi_2\in\operatorname{Irr}_{\mathrm{gen}}(G_n)$.
We say that $s=s_0$ is an \emph{exceptional pole with level $m$}
for $(\pi_1,\pi_2)$ if it satisfies the equivalent conditions of
Theorem~\ref{1.1}.
\end{defn}

For the remainder of this section, fix a pair
$(\pi_1,\pi_2)$ as in Theorem~\ref{classification-general}.
We say that $z\in\mathbb C$ is exceptional for this pair if it is
exceptional with level $m$ for some integer $m\geq0$, and we put
\[
P_{\mathrm{ex}}
=
\{z\in\mathbb C\,|\,\text{$z$ is exceptional for $(\pi_1,\pi_2)$}\}.
\]

By Proposition~\ref{pole}, every element of $P_{\mathrm{ex}}$ is a pole of
$L(s,\pi_1\times\pi_2)$. 

For a meromorphic function $F$ and $z\in\mathbb C$, we define
$\operatorname{pord}_z(F)$ to be the order of the pole of $F$ at
$s=z$, with the convention that $\operatorname{pord}_z(F)=0$ if
$F$ is holomorphic at $z$. Thus, if $F$ has a pole at $z$ and
\[
F(s)=(s-z)^{-d}u(s),
\qquad d\geq1,
\]
where $u$ is holomorphic near $z$ and $u(z)\neq0$, then
$\operatorname{pord}_z(F)=d$. 

For any $z \in \mathbb{C}$, we then define
\[
d_{\mathrm{ex}}(z;\pi_1,\pi_2) = 
\begin{cases} 
\operatorname{pord}_{z} \left(L(s,\pi_1\times\pi_2)\right) & \text{if } z \in P_{\mathrm{ex}}, \\ 
0 & \text{if } z \notin P_{\mathrm{ex}}. 
\end{cases}
\]

\begin{prop}\label{exceptional-shift}
If $z \in P_{\mathrm{ex}}$, then $z-2 \in P_{\mathrm{ex}}$ and
$$d_{\mathrm{ex}}(z-2;\pi_1,\pi_2) \geq d_{\mathrm{ex}}(z;\pi_1,\pi_2).$$
Equivalently, for every $z \in \mathbb{C}$,
$$d_{\mathrm{ex}}(z;\pi_1,\pi_2) \geq d_{\mathrm{ex}}(z+2;\pi_1,\pi_2).$$
\end{prop}

\begin{proof}
Suppose $z \in P_{\mathrm{ex}}$. By Theorem~\ref{classification-general}, there exist integers $a_1, \ldots, a_{r+l} \geq 0$ summing to $m$, and integers $p_v \geq 2$ $(1 \leq v \leq l)$ satisfying
$$|p_v-k_v| \leq a_{r+v} \quad \text{and} \quad p_v \equiv k_v+a_{r+v} \pmod{2},$$
such that, after reordering the inducing blocks,
$$\pi_1 \simeq \left( \bigtimes_{u=1}^{r} \chi_{\epsilon_u+a_u,\,-(s_u+a_u+z)} \right) \times \left( \bigtimes_{v=1}^{l} D_{p_v,\,-\left(t_v+\frac{a_{r+v}}{2}+z\right)} \right).$$

Define the shifted integers
$$a'_u = a_u+2 \quad (1 \leq u \leq r) \quad \text{and} \quad a'_{r+v} = a_{r+v}+4 \quad (1 \leq v \leq l).$$
We then have
$$\sum_{u=1}^{r+l} a'_u = m + 2r + 4l = m + 2n.$$
Furthermore, these new integers preserve the required parities and shifts:
$$\begin{aligned}
\epsilon_u+a'_u &\equiv \epsilon_u+a_u \pmod{2}, \\
s_u+a'_u+(z-2) &= s_u+a_u+z, \\
t_v+\frac{a'_{r+v}}{2}+(z-2) &= t_v+\frac{a_{r+v}}{2}+z.
\end{aligned}$$
The inequalities and parity conditions on $p_v$ also remain valid since
$$|p_v-k_v| \leq a_{r+v} < a'_{r+v}$$
and
$$p_v \equiv k_v+a_{r+v} \equiv k_v+a'_{r+v} \pmod{2}.$$
Hence, Theorem~\ref{classification-general} shows that $z-2 \in P_{\mathrm{ex}}$.

By Proposition~\ref{pole}, the pole orders at $z$ and $z-2$ are given by
$$d_{\mathrm{ex}}(z;\pi_1,\pi_2) = r+l + \#\{1 \leq v \leq l \mid a_{r+v} \geq p_v+k_v-2\}$$
and
$$d_{\mathrm{ex}}(z-2;\pi_1,\pi_2) = r+l + \#\{1 \leq v \leq l \mid a'_{r+v} \geq p_v+k_v-2\},$$
respectively. Since $a'_{r+v} = a_{r+v}+4 > a_{r+v}$, the set of indices satisfying the inequality for $z-2$ contains the set for $z$. Therefore,
$$d_{\mathrm{ex}}(z-2;\pi_1,\pi_2) \geq d_{\mathrm{ex}}(z;\pi_1,\pi_2).$$
This completes the proof.
\end{proof}

\begin{corollary}\label{exceptional-set}
Suppose that $P_{\mathrm{ex}}\neq\varnothing$. Then there exist
finitely many complex numbers $z_1,\ldots,z_q$, pairwise incongruent
modulo $2\mathbb Z$, such that
\[
P_{\mathrm{ex}}
=
\bigsqcup_{i=1}^{q}
\left\{
z_i-2k
\;\middle|\;
k\in\mathbb Z_{\geq0}
\right\}.
\]
\end{corollary}

\begin{proof}
The Rankin--Selberg factor $L(s,\pi_1\times\pi_2)$ is a finite product
of shifted $\Gamma_{\mathbb R}$- and $\Gamma_{\mathbb C}$-factors.
Since
\[
\Gamma_{\mathbb C}(s)
=
\Gamma_{\mathbb R}(s)\Gamma_{\mathbb R}(s+1),
\]
its pole set is contained in a finite union of arithmetic progressions
of the form $b-2\mathbb Z_{\geq0}.$
By Proposition~\ref{pole}, every exceptional pole is a pole of
$L(s,\pi_1\times\pi_2)$. Hence $P_{\mathrm{ex}}$ is contained in a
finite union of such progressions and therefore meets only finitely
many congruence classes modulo $2\mathbb Z$.

In each congruence class modulo $2\mathbb Z$ that meets
$P_{\mathrm{ex}}$, choose an exceptional point $z_i$ with maximal real
part. Such a point exists because $P_{\mathrm{ex}}$ is contained in a
finite union of the above progressions. By
Proposition~\ref{exceptional-shift}, every point $z_i-2k$, with
$k\in\mathbb Z_{\geq0}$, is exceptional. Conversely, if $z$ is an
exceptional point congruent to $z_i$ modulo $2\mathbb Z$, then
$z=z_i+2k$ for some $k\in\mathbb Z$. The maximality of
$\operatorname{Re}(z_i)$ implies that $k\leq0$. Therefore,
\[
P_{\mathrm{ex}}\cap(z_i+2\mathbb Z)
=
\left\{
z_i-2k
\;\middle|\;
k\in\mathbb Z_{\geq0}
\right\}.
\]
The resulting progressions are pairwise disjoint because the $z_i$
are pairwise incongruent modulo $2\mathbb Z$.
\end{proof}

For $z\in\mathbb C$, put
\[
c_{\mathrm{ex}}(z;\pi_1,\pi_2)
=
d_{\mathrm{ex}}(z;\pi_1,\pi_2)
-
d_{\mathrm{ex}}(z+2;\pi_1,\pi_2).
\]

\begin{lemma}\label{exceptional-jumps}
For every $z\in\mathbb C$, one has
\[
c_{\mathrm{ex}}(z;\pi_1,\pi_2)\in\mathbb Z_{\geq0}.
\]
Moreover, $c_{\mathrm{ex}}(z;\pi_1,\pi_2)\neq0$ for only finitely many
$z\in\mathbb C$.
\end{lemma}

\begin{proof}
The nonnegativity follows immediately from
Proposition~\ref{exceptional-shift}. Suppose that $P_{\mathrm{ex}}\neq\varnothing$, and write
\[
P_{\mathrm{ex}}
=
\bigsqcup_{i=1}^{q}
\left\{
z_i-2k
\;\middle|\;
k\in\mathbb Z_{\geq0}
\right\}
\]
as in Corollary~\ref{exceptional-set}. For each $i$, the sequence
\[
d_{\mathrm{ex}}(z_i;\pi_1,\pi_2),\ 
d_{\mathrm{ex}}(z_i-2;\pi_1,\pi_2),\ 
d_{\mathrm{ex}}(z_i-4;\pi_1,\pi_2),\ldots
\]
is nondecreasing by Proposition~\ref{exceptional-shift}. It is a
sequence of integers bounded above by $n$, by
Proposition~\ref{pole}, and hence is eventually constant. Therefore,
\[
c_{\mathrm{ex}}(z_i-2k;\pi_1,\pi_2)=0
\]
for all sufficiently large $k$.

Moreover, if $z\notin P_{\mathrm{ex}}$, then $d_{\mathrm{ex}}(z;\pi_1,\pi_2)=0$; the
inequality
\[
d_{\mathrm{ex}}(z;\pi_1,\pi_2)
\geq
d_{\mathrm{ex}}(z+2;\pi_1,\pi_2)
\]
also gives $d_{\mathrm{ex}}(z+2;\pi_1,\pi_2)=0$, and hence
$c_{\mathrm{ex}}(z;\pi_1,\pi_2)=0$. Since $P_{\mathrm{ex}}$ is a finite union of arithmetic progressions,
it follows that $c_{\mathrm{ex}}$ has finite support.
\end{proof}

We now define the exceptional $L$-factor.

\begin{defn}\label{Lex}
The \emph{exceptional $L$-factor} of $\pi_1$ and $\pi_2$ is defined by
\[
L_{\mathrm{ex}}(s,\pi_1\times\pi_2)
=
\prod_{\substack{z\in\mathbb C\\
c_{\mathrm{ex}}(z;\pi_1,\pi_2)>0}}
\Gamma_{\mathbb R}(s-z)^{
c_{\mathrm{ex}}(z;\pi_1,\pi_2)}.
\]
By Lemma~\ref{exceptional-jumps}, this is a finite product. If
$P_{\mathrm{ex}}=\varnothing$, the product is empty and we put
\[
L_{\mathrm{ex}}(s,\pi_1\times\pi_2)=1.
\]
\end{defn}

\begin{prop}\label{Lex-pole-orders}
For every $w\in\mathbb C$,
\[
\operatorname{pord}_{w}
L_{\mathrm{ex}}(s,\pi_1\times\pi_2)
=
d_{\mathrm{ex}}(w;\pi_1,\pi_2).
\]
Consequently, the poles of
$L_{\mathrm{ex}}(s,\pi_1\times\pi_2)$ are precisely the exceptional poles of
the pair $(\pi_1,\pi_2)$, and their orders agree with the corresponding
pole orders of $L(s,\pi_1\times\pi_2)$.
\end{prop}

\begin{proof}
The factor $\Gamma_{\mathbb R}(s-z)$ has simple poles at
\[
z,\ z-2,\ z-4,\ldots.
\]
Therefore,
\[
\operatorname{pord}_{w}
L_{\mathrm{ex}}(s,\pi_1\times\pi_2)
=
\sum_{k\geq0}
c_{\mathrm{ex}}(w+2k;\pi_1,\pi_2).
\]
The sum is finite by Lemma~\ref{exceptional-jumps}. Using the
definition of $c_{\mathrm{ex}}$, it telescopes:
\[
\begin{aligned}
\sum_{k\geq0}
c_{\mathrm{ex}}(w+2k;\pi_1,\pi_2)
=
\sum_{k\geq0}
\Bigl(
d_{\mathrm{ex}}(w+2k;\pi_1,\pi_2)
-
d_{\mathrm{ex}}(w+2k+2;\pi_1,\pi_2)
\Bigr)
=
d_{\mathrm{ex}}(w;\pi_1,\pi_2).
\end{aligned}
\]
Indeed, by Corollary~\ref{exceptional-set},
$d_{\mathrm{ex}}(w+2k;\pi_1,\pi_2)=0$ for all sufficiently large $k$. This
proves the assertion.
\end{proof}

Next, we define the least common multiple for finite products of inverse $\Gamma_{\mathbb R}$-factors. Consider a finite set $\mathcal{F}=\{F_1(s),\ldots,F_r(s)\}$, where
$$F_i(s) = \prod_j \Gamma_{\mathbb R}(s-a_{ij})^{-m_{ij}}$$
for some constants $a_{ij}\in\mathbb C$ and integers $m_{ij}\in\mathbb Z_{\geq0}$. Because $\Gamma_{\mathbb R}(s)^{-1}$ is an entire function, each $F_i$ is also entire. For $z\in\mathbb C$, let $\operatorname{ord}_z(F_i)$ denote the order of vanishing of $F_i$ at $s=z$. We then define
$$M_{\mathcal F}(z) = \max_{1\leq i\leq r} \operatorname{ord}_z(F_i) \quad \text{and} \quad c_{\mathcal F}(z) = M_{\mathcal F}(z)-M_{\mathcal F}(z+2).$$

\begin{lemma}\label{lem:lcm-jumps}
For every $z\in\mathbb C$, we have $c_{\mathcal F}(z)\in\mathbb Z_{\geq0}$. Moreover, $c_{\mathcal F}(z)\neq0$ for only finitely many $z\in\mathbb C$.
\end{lemma}

\begin{proof}
For any given $i$ and $z\in\mathbb C$, the order of vanishing is
$$\operatorname{ord}_z(F_i) = \sum_{\substack{j\\ a_{ij}-z\in2\mathbb Z_{\geq0}}} m_{ij}.$$
Consequently, the difference between the orders at $z$ and $z+2$ simplifies to
$$\operatorname{ord}_z(F_i) - \operatorname{ord}_{z+2}(F_i) = \sum_{\substack{j\\a_{ij}=z}} m_{ij} \geq 0.$$
Since $\operatorname{ord}_z(F_i) \geq \operatorname{ord}_{z+2}(F_i)$ holds for all $i$, taking the maximum over $1 \leq i \leq r$ yields $M_{\mathcal F}(z)\geq M_{\mathcal F}(z+2)$, which proves that $c_{\mathcal F}(z)\in\mathbb Z_{\geq0}$.

Finally, if $z\notin\{a_{ij}\}_{i,j}$, then $\operatorname{ord}_z(F_i) = \operatorname{ord}_{z+2}(F_i)$ for all $i$. This implies $M_{\mathcal F}(z)=M_{\mathcal F}(z+2)$, and hence $c_{\mathcal F}(z)=0$. Therefore, the support of $c_{\mathcal F}$ satisfies $\operatorname{Supp}(c_{\mathcal F}) \subseteq \{a_{ij}\}_{i,j}$, which is a finite set.
\end{proof}

\begin{defn}\label{lcm}
The \emph{least common multiple} of
$\mathcal F=\{F_1,\ldots,F_r\}$ is defined by
\[
\operatorname{l.c.m.}(F_1,\ldots,F_r)
=
\prod_{\substack{z\in\mathbb C\\c_{\mathcal F}(z)>0}}
\Gamma_{\mathbb R}(s-z)^{-c_{\mathcal F}(z)}.
\]
By Lemma~\ref{lem:lcm-jumps}, this is a finite product.
\end{defn}

Throughout, all finite products of inverse
$\Gamma_{\mathbb R}$-factors are understood to have scalar
prefactor equal to $1$.

\begin{prop}[Characterization of the l.c.m.]
\label{prop:lcm-characterization}
Let
\[
\mathcal F=\{F_1,\ldots,F_r\}
\]
be as above, and put
\[
H(s)=\operatorname{l.c.m.}(F_1,\ldots,F_r).
\]
Then, for every $w\in\mathbb C$,
\[
\operatorname{ord}_w(H)
=
\max_{1\leq i\leq r}\operatorname{ord}_w(F_i).
\]
In particular, $H/F_i$ is entire for every $i$. Moreover, if $K$ is
another finite product of inverse $\Gamma_{\mathbb R}$-factors such
that $K/F_i$ is entire for every $i$, then $K/H$ is entire. Finally, $H$ is the unique finite product of inverse
$\Gamma_{\mathbb R}$-factors satisfying the displayed equality of
orders.
\end{prop}

\begin{proof}
By Definition~\ref{lcm},
\[
H(s)
=
\prod_{\substack{z\in\mathbb C\\c_{\mathcal F}(z)>0}}
\Gamma_{\mathbb R}(s-z)^{-c_{\mathcal F}(z)}.
\]
Therefore,
\[
\begin{aligned}
\operatorname{ord}_w(H)
=
\sum_{k\geq0}c_{\mathcal F}(w+2k)
=
\sum_{k\geq0}
\left(
M_{\mathcal F}(w+2k)
-
M_{\mathcal F}(w+2k+2)
\right)
=
M_{\mathcal F}(w)
=
\max_{1\leq i\leq r}\operatorname{ord}_w(F_i).
\end{aligned}
\]
Indeed, for fixed $w$, the explicit formula for
$\operatorname{ord}_{w+2k}(F_i)$ shows that
$M_{\mathcal F}(w+2k)=0$ for all sufficiently large $k$.
Hence the sum telescopes to $M_{\mathcal F}(w)$. It follows that
$H/F_i$ is entire for every $i$.

If $K/F_i$ is entire for every $i$, then
\[
\operatorname{ord}_w(K)
\geq
\max_{1\leq i\leq r}\operatorname{ord}_w(F_i)
=
\operatorname{ord}_w(H)
\]
for every $w\in\mathbb C$. Hence $K/H$ is entire.

It remains to prove the uniqueness assertion. Let $G$ be a finite
product of inverse $\Gamma_{\mathbb R}$-factors. After combining equal
shifts, write
\[
G(s)
=
\prod_{z\in S_G}
\Gamma_{\mathbb R}(s-z)^{-m_G(z)},
\]
where $S_G\subset\mathbb C$ is finite and
$m_G(z)\in\mathbb Z_{>0}$. We extend $m_G$ by zero outside $S_G$.
Since $\Gamma_{\mathbb R}(s-z)^{-1}$ has simple zeros at
\[
z,\ z-2,\ z-4,\ldots,
\]
we have
\[
\operatorname{ord}_w(G)
=
\sum_{k\geq0}m_G(w+2k).
\]
Consequently,
\[
m_G(w)
=
\operatorname{ord}_w(G)
-
\operatorname{ord}_{w+2}(G).
\]
Thus the orders of vanishing of $G$ determine the multiplicity of
every inverse $\Gamma_{\mathbb R}$-factor occurring in $G$. Therefore, if another finite product of inverse
$\Gamma_{\mathbb R}$-factors has the same orders of vanishing as $H$
at every point of $\mathbb C$, then it has the same factor
multiplicities as $H$ and is equal to $H$.
\end{proof}
\begin{Example}
Consider the functions
\begin{align*}
F_1(s)
&=
\Gamma_{\mathbb R}(s)^{-1}
\Gamma_{\mathbb R}(s-2)^{-2},\\
F_2(s)
&=
\Gamma_{\mathbb R}(s)^{-2}
\Gamma_{\mathbb R}(s-4)^{-1},\\
F_3(s)
&=
\Gamma_{\mathbb R}(s-2)^{-1}
\Gamma_{\mathbb R}(s-4)^{-2}.
\end{align*}
The orders of vanishing at the relevant points are
\[
\begin{array}{c|ccc|c}
z
&
\operatorname{ord}_z(F_1)
&
\operatorname{ord}_z(F_2)
&
\operatorname{ord}_z(F_3)
&
M_{\mathcal F}(z)
\\ \hline
4  & 0 & 1 & 2 & 2\\
2  & 2 & 1 & 3 & 3\\
0  & 3 & 3 & 3 & 3\\
-2 & 3 & 3 & 3 & 3
\end{array}
\]
and the same maximum value $3$ occurs at
$-4,-6,\ldots$. Also,
\[
M_{\mathcal F}(6)=0.
\]
It follows that
\[
c_{\mathcal F}(4)
=
M_{\mathcal F}(4)-M_{\mathcal F}(6)
=
2,
\]
and
\[
c_{\mathcal F}(2)
=
M_{\mathcal F}(2)-M_{\mathcal F}(4)
=
1.
\]
At every other point,
\[
c_{\mathcal F}(z)=0.
\]
Therefore,
\[
\operatorname{l.c.m.}(F_1,F_2,F_3)
=
\Gamma_{\mathbb R}(s-4)^{-2}
\Gamma_{\mathbb R}(s-2)^{-1}.
\]
Indeed, this function has a zero of order $2$ at $s=4$ and zeros of
order $3$ at
\[
s=2,0,-2,\ldots,
\]
which are precisely the maximum orders of vanishing of
$F_1,F_2,F_3$ at these points.
\end{Example}

Next, we recall the notion of archimedean derivative introduced in
\cite{Chai2015}. For \(1\leq l\leq n\), define
\[
U_{n-l+1}
=
\left\{
\begin{pmatrix}
I_{n-l} & x\\
0 & u
\end{pmatrix}
\;\middle|\;
x\in\operatorname{Mat}_{n-l,l}(\mathbb R),\
u\in N_l
\right\}.
\]
For \(l<n\), this is the unipotent radical of the standard parabolic
subgroup associated with the partition
\((n-l,1,\ldots,1)\). For \(l=n\), the upper-left block is absent and
\(U_1=N_n\). We regard \(G_0\) as the trivial group. Let \(\mathfrak u_{n-l+1}\) denote the Lie algebra of \(U_{n-l+1}\).

Let $\mu$ be the differential of $\psi$. Thus, $\mu$ is a linear
form on $\mathfrak n_n$, the Lie algebra of $N_n$, which vanishes on
$[\mathfrak n_n,\mathfrak n_n]$. For
$X\in\mathfrak u_{n-l+1}$, define
\[
\mu_{n-l+1}(X)
=
\mu\left(
\sum_{i=n-l+1}^{n-1}
X_{i,i+1}E_{i,i+1}
\right),
\]
where $E_{i,j}$ denotes the standard matrix unit.

\begin{defn}[Archimedean derivative]
\label{def:archimedean-derivative}
Let $(\pi,V)\in\Irr(G_n)$. For every integer $0\leq l\leq n$, the
$l$-th derivative
\[
(\pi^{(l)},V^{(l)})
\]
is defined as follows:
\begin{enumerate}
    \item If $l=0$, put
    \[
    (\pi^{(0)},V^{(0)})=(\pi,V).
    \]

    \item If $1\leq l\leq n$, let $V_l$ be the closure of the
    subspace of $V$ spanned by
    \[
    X\cdot v-\mu_{n-l+1}(X)v,
    \qquad
    X\in\mathfrak u_{n-l+1},\quad v\in V,
    \]
    and put
    \[
    V^{(l)}=V/V_l.
    \]
    For $g\in G_{n-l}$, define
    \[
    \pi^{(l)}(g)(v+V_l)
    =
    |\det(g)|^{-l/2}
    \pi\left(
    \begin{pmatrix}
    g&0\\
    0&I_l
    \end{pmatrix}
    \right)v+V_l.
    \]
\end{enumerate}
\end{defn}

\begin{remark}[Derivative as a Whittaker quotient of homology]
\label{rem:derivative-homology}
For a smooth representation $V$, we use the Hausdorff quotient
\[
H_0(\mathfrak u,V)
=
V\big/
\overline{\operatorname{Span}\{X\cdot v\,|\,X\in\mathfrak u,\ v\in V\}},
\]
with its quotient topology. For a Harish--Chandra module $M$, we use
the algebraic quotient $H_0(\mathfrak u,M)=M/\mathfrak u M$.
These are the conventions of \cite[Section~2]{Chai2015}.

Let $1\leq l\leq n-1$, and let
\[
Q_l=P_{n-l,l}=M_lN_{n-l,l},
\qquad
M_l\simeq G_{n-l}\times G_l,
\]
where $N_{n-l,l}$ is the unipotent radical of $Q_l$. Let $\psi_l$ denote the standard nondegenerate character of
$N_l$. For a representation $X$ of $G_{n-l}\times G_l$, write
\[
X_{N_l,\psi_l}
=
X\bigg/
\overline{
\operatorname{Span}
\left\{
u\cdot x-\psi_l(u)x
\;\middle|\;
u\in N_l,\ x\in X
\right\}
}.
\]
Then Chai's realization of the derivative gives
\begin{equation}
\label{eq:derivative-homology-realization}
\pi^{(l)}
\cong
|\det|^{-l/2}\otimes
\left(
H_0(\mathfrak n_{n-l,l},\pi)
\right)_{N_l,\psi_l},
\end{equation}
where the determinant character acts on the $G_{n-l}$-factor.
\end{remark}

The $1^{st}$ derivative is essentially computed by Chai  in his Ph.D. thesis (cf.\cite{Chai2015}). 
We compute all higher order derivatives.

\begin{prop}[Explicit derivatives in general position]
\label{explicit-derivative}
Let
\[
\pi=
\left(
\bigtimes_{i=1}^{r}\chi_{\epsilon_i,s_i}
\right)
\times
\left(
\bigtimes_{j=1}^{\ell}D_{k_j,t_j}
\right),
\qquad
n=r+2\ell,
\]
be an irreducible generic representation of $G_n$ in general
position. For $1\leq j\leq\ell$, put $h_j=\frac{k_j-1}{2}$.

For $0\leq q\leq n$, let $\mathscr I_q$ denote the set of triples
$\iota=(\mathbf a,\mathbf b,\boldsymbol\delta)$, where
\[
\begin{gathered}
\mathbf a=(a_1,\ldots,a_r)\in\{0,1\}^r,
\qquad
\mathbf b=(b_1,\ldots,b_\ell)\in\{0,1,2\}^{\ell},\\
\sum_{i=1}^{r}a_i+\sum_{j=1}^{\ell}b_j=q.
\end{gathered}
\]
Here $\boldsymbol\delta=(\delta_j)_{\{j:b_j=1\}}$, with
$\delta_j\in\{0,1\}$ for every $j$ such that $b_j=1$.

For $\iota=(\mathbf a,\mathbf b,\boldsymbol\delta)\in\mathscr I_q$, define a representation of $G_{n-q}$ by
\[
\pi_\iota
=
\left(
\bigtimes_{\substack{1\leq i\leq r\\a_i=0}}
\chi_{\epsilon_i,s_i}
\right)
\times
\left(
\bigtimes_{\substack{1\leq j\leq\ell\\b_j=0}}
D_{k_j,t_j}
\right)
\times
\left(
\bigtimes_{\substack{1\leq j\leq\ell\\b_j=1}}
\chi_{\delta_j,t_j+h_j}
\right).
\]
Then
\[
\boxed{
\pi^{(q)}
\cong
\bigoplus_{\iota\in\mathscr I_q}\pi_\iota.
}
\]
For \(0\leq q<n\), every representation \(\pi_\iota\) occurring in
the above direct sum is irreducible, generic, and in general position. The direct
sum is indexed by $\mathscr I_q$, so that repetitions, if any, are
retained. Empty products are interpreted as the one-dimensional
representation of $G_0$.
\end{prop}

\begin{proof}
For $q=0$, the set $\mathscr I_0$ has one element, with $a_i=0$ and
$b_j=0$ for all $i$ and $j$, and the assertion is simply
$\pi^{(0)}=\pi$. For $q=n$, the set $\mathscr I_n$ also has one element,
with $a_i=1$ and $b_j=2$ for all $i$ and $j$. Since $\pi$ is irreducible
and generic, its maximal Hausdorff Whittaker quotient is
one-dimensional, and hence $\pi^{(n)}\simeq\mathbb C$.

Assume that $1\leq q\leq n-1$. Put
\[
Q=P_{n-q,q}=L_QU_Q,
\qquad
L_Q\cong G_{n-q}\times G_q,
\qquad
\mathfrak u_Q=\operatorname{Lie}(U_Q),
\]
and let $\pi_K$ be the Harish--Chandra module of $\pi$. Set
\[
\lambda=
\bigl(
 s_1,\ldots,s_r,
 t_1+h_1,t_1-h_1,\ldots,
 t_\ell+h_\ell,t_\ell-h_\ell
\bigr).
\]
For a parity tuple
\[
\eta=
\bigl(
\epsilon_1,\ldots,\epsilon_r,
\eta_1^+,\eta_1^-,\ldots,
\eta_\ell^+,\eta_\ell^-
\bigr)
\]
satisfying
\[
\eta_j^++\eta_j^-\equiv k_j\pmod2
\qquad (1\leq j\leq\ell),
\]
write
\[
\pi(\lambda;\eta)
=
\bigboxtimes_{\alpha=1}^{n}
\chi_{\eta_\alpha,\lambda_\alpha}.
\]
Let $W=S_n$ act by simultaneously permuting the exponent and parity
entries, and put
\[
\mathcal P
=
W\cdot
\left\{
\pi(\lambda;\eta)\,|\,
\eta_j^++\eta_j^-\equiv k_j\pmod2
\text{ for every }j
\right\}.
\]
Let $\mathcal P_{\emptyset}$ be the subset consisting of those
$\pi(\mu;\nu)$ for which $t_j+h_j$ precedes $t_j-h_j$ in $\mu$ for
every $j$.

For $\pi(\mu;\nu)\in\mathcal P_{\emptyset}$, let
\[
I^{L_Q}(\mu;\nu)
=
\left(
\bigtimes_{\alpha=1}^{n-q}
\chi_{\nu_\alpha,\mu_\alpha}
\right)
\widehat\otimes
\left(
\bigtimes_{\alpha=n-q+1}^{n}
\chi_{\nu_\alpha,\mu_\alpha}
\right),
\]
write $I^{L_Q}(\mu;\nu)_K$ for its Harish--Chandra module, and let
$J^{L_Q}(\mu;\nu)$ be its unique irreducible submodule. Let
$\sim_Q$ be the equivalence relation of \cite[Section~4]{CC1999},
whose specialization to the present parameters is described below.
By \cite[Theorem~4.2]{CC1999},
\begin{equation}\label{eq:CC-decomposition-explicit}
H_0(\mathfrak u_Q,\pi_K)
\cong
\bigoplus_{[\pi(\mu;\nu)]\in
\mathcal P_{\emptyset}/\!\sim_Q}
J^{L_Q}(\mu;\nu)\otimes\delta_Q^{1/2}.
\end{equation}

The entries of \(\lambda\) are pairwise distinct. More precisely, the
only pairs of entries whose difference belongs to \(\mathbb Z\) are
\[
t_j+h_j,\qquad t_j-h_j
\]
arising from the same discrete-series block; their difference is
\[
(t_j+h_j)-(t_j-h_j)=2h_j=k_j-1.
\]
Indeed, all differences between entries arising from distinct original
inducing blocks are nonintegral by the general position hypothesis.
Specializing the definition of \(\sim_Q\) in
\cite[Section~4]{CC1999} to these parameters gives the following
generating operations:

\begin{enumerate}[label=\textnormal{(\roman*)}]
\item\label{operation1} simultaneous permutations of the exponent and parity entries
within either Levi block, provided that the resulting parameters remain
in $\mathcal P_{\emptyset}$;

\item\label{operation2} for an original string $t_j+h_j,t_j-h_j$ contained entirely in
one Levi block, the parity change
\[
(\eta_j^+,\eta_j^-)
\longmapsto
(\eta_j^++1,\eta_j^-+1)
\pmod2.
\]
\end{enumerate}

If the two entries of a string lie in different Levi blocks, their
parities are not identified by $\sim_Q$.

We now identify the quotient
$\mathcal P_{\emptyset}/\!\sim_Q$. For a class represented by
$\pi(\mu;\nu)$, define
\[
a_i=
\begin{cases}
0,&s_i\text{ occurs in the first Levi block},\\
1,&s_i\text{ occurs in the second Levi block},
\end{cases}
\]
and
\[
b_j=
\#\left\{
t_j+h_j,\ t_j-h_j
\text{ occurring in the second Levi block}
\right\}.
\]
Thus $a_i\in\{0,1\}$ and $b_j\in\{0,1,2\}$. Moreover,
\[
\sum_{i=1}^{r}a_i+\sum_{j=1}^{\ell}b_j=q.
\]

If $b_j=1$, the condition defining $\mathcal P_{\emptyset}$ forces
$t_j+h_j$ to occur in the first Levi block and $t_j-h_j$ in the
second. In this case, let $\delta_j$ be the parity attached to
$t_j+h_j$. We therefore obtain a map
\[
\Phi:
\mathcal P_{\emptyset}/\!\sim_Q
\longrightarrow
\mathscr I_q,
\qquad
[\pi(\mu;\nu)]
\longmapsto
(\mathbf a,\mathbf b,\boldsymbol\delta).
\]

\smallskip
\noindent\emph{Well-definedness.}
The relation $\sim_Q$ does not move an exponent entry from one Levi
block to the other and does not change the parity attached to a split
string. The parity choice for an unsplit string is not recorded in
$\mathscr I_q$. Hence $\Phi$ is well-defined.

\smallskip
\noindent\emph{Surjectivity.}
Let
\[
\iota=(\mathbf a,\mathbf b,\boldsymbol\delta)\in\mathscr I_q.
\]
Place $s_i$ in the first or second Levi block according as $a_i=0$ or
$a_i=1$. Distribute the entries of the $j$-th two-term string according
to the following table:
\[
\begin{array}{c|c|c}
b_j
&
\text{first Levi block}
&
\text{second Levi block}
\\ \hline
0
&
t_j+h_j,\ t_j-h_j
&
\text{none}
\\[1mm]
1
&
t_j+h_j
&
t_j-h_j
\\[1mm]
2
&
\text{none}
&
t_j+h_j,\ t_j-h_j
\end{array}
\]

When $b_j=1$, give $t_j+h_j$ parity $\delta_j$ and give
$t_j-h_j$ the unique parity $\delta_j^-\in\{0,1\}$ satisfying
\[
\delta_j+\delta_j^-\equiv k_j\pmod2.
\]
When $b_j=0$ or $b_j=2$, choose either admissible parity pair.

The defining rank condition for $\mathscr I_q$ ensures that the second
Levi block contains exactly $q$ entries. Ordering the entries within
each block so that $t_j+h_j$ precedes $t_j-h_j$ whenever both occur in
that block gives an element of $\mathcal P_{\emptyset}$ whose image
under $\Phi$ is $\iota$. Thus $\Phi$ is surjective.

\smallskip
\noindent\emph{Injectivity.}
Suppose that two classes have the same image under $\Phi$. They then
contain the same exponent entries in each Levi block, and the parities
on all split strings agree. Using operation \ref{operation1}, we may
choose representatives whose entries occur in the same order within
each block. Using operation \ref{operation2}, we may then make the
parity pairs on every unsplit string agree. The resulting parameters
are equal, so the original classes are $\sim_Q$-equivalent.

Consequently, $\Phi$ is a bijection, and hence
\begin{equation}\label{eq:CC-Iq-bijection-correct}
\mathcal P_{\emptyset}/\!\sim_Q
\ \longleftrightarrow\
\mathscr I_q.
\end{equation}

Fix \(\iota=(\mathbf a,\mathbf b,\boldsymbol\delta)\in\mathscr I_q\).
For each \(j\) with \(b_j=1\), let \(\delta_j^-\) be as above.
Put \(\rho_\iota=\pi_\iota\), and define
\[
\sigma_\iota
=
\left(
\bigtimes_{\substack{1\leq i\leq r\\a_i=1}}
\chi_{\epsilon_i,s_i}
\right)
\times
\left(
\bigtimes_{\substack{1\leq j\leq\ell\\b_j=2}}
D_{k_j,t_j}
\right)
\times
\left(
\bigtimes_{\substack{1\leq j\leq\ell\\b_j=1}}
\chi_{\delta_j^-,t_j-h_j}
\right).
\]
The only new character exponents in $\rho_\iota$ and $\sigma_\iota$
are $t_j\pm h_j$. Since $h_j\in\frac12\mathbb Z$ and
$t_j\notin\frac12\mathbb Z$, these exponents do not lie in
$\mathbb Z$. Moreover, for distinct original blocks,
\[
\begin{aligned}
 s_i-(t_j\pm h_j)&\notin\tfrac12\mathbb Z,\\
 (t_j\pm h_j)-t_v&\notin\tfrac12\mathbb Z,\\
 (t_j\pm h_j)-(t_v\pm h_v)&\notin\tfrac12\mathbb Z.
\end{aligned}
\]
All remaining conditions are inherited from $\pi$. Thus
$\rho_\iota$ and $\sigma_\iota$ are in general position, and hence are
irreducible and generic.

Choose a representative $\pi(\mu_\iota;\nu_\iota)$ of the class
corresponding to $\iota$ in which the two entries of every unsplit
string are adjacent in their Levi block and occur in the order
$t_j+h_j,t_j-h_j$. For every admissible parity pair
\(\delta+\delta^-\equiv k_j\pmod2\), taking contragredients in
\eqref{ses:Ik} and applying the appropriate determinant twist
gives an embedding of Harish--Chandra modules
\vspace{-.1cm}
\[
\begin{tikzcd}
(D_{k_j,t_j})_{K_2}
\arrow[r, hook] &
\left(
\chi_{\delta,t_j+h_j}
\times
\chi_{\delta^-,t_j-h_j}
\right)_{K_2}.
\end{tikzcd}
\vspace{-.1cm}
\]
Taking these embeddings for all unsplit strings and using exactness of
normalized parabolic induction gives
\vspace{-.1cm}
\[
\begin{tikzcd}
(\rho_\iota)_{K_{n-q}}
\boxtimes
(\sigma_\iota)_{K_q}
\arrow[r, hook] &
I^{L_Q}(\mu_\iota;\nu_\iota)_K.
\end{tikzcd}
\vspace{-.1cm}
\]
The source is irreducible, while the target has the unique irreducible
submodule $J^{L_Q}(\mu_\iota;\nu_\iota)$. Consequently,
\[
J^{L_Q}(\mu_\iota;\nu_\iota)
\cong
(\rho_\iota)_{K_{n-q}}
\boxtimes
(\sigma_\iota)_{K_q}.
\]

Combining this with
\eqref{eq:CC-decomposition-explicit} and
\eqref{eq:CC-Iq-bijection-correct}, and then passing to
Casselman--Wallach globalizations, gives
\[
H_0(\mathfrak u_Q,\pi)
\cong
\bigoplus_{\iota\in\mathscr I_q}
(\rho_\iota\widehat\otimes\sigma_\iota)
\otimes\delta_Q^{1/2},
\]
where \(\delta_Q\) denotes the modulus character of $Q$. Here we use the compatibility of degree-zero nilradical homology with
Casselman--Wallach globalization
\cite[Proposition~2.2]{Chai2015}.

For $g_1\in G_{n-q}$ and $g_2\in G_q$,
\[
\delta_Q^{1/2}
\bigl(\operatorname{diag}(g_1,g_2)\bigr)
=
|\det(g_1)|^{q/2}|\det(g_2)|^{-(n-q)/2}.
\]
Hence
\[
H_0(\mathfrak u_Q,\pi)
\cong
\bigoplus_{\iota\in\mathscr I_q}
\left(\rho_\iota\otimes|\det|^{q/2}\right)
\widehat\otimes
\left(\sigma_\iota\otimes|\det|^{-(n-q)/2}\right).
\]
Since $\sigma_\iota$ is irreducible and generic, its maximal Hausdorff
$(N_q,\psi_q)$-quotient is one-dimensional. Determinant characters are
trivial on $N_q$, and therefore
\[
\left(H_0(\mathfrak u_Q,\pi)\right)_{N_q,\psi_q}
\cong
\bigoplus_{\iota\in\mathscr I_q}
\left(\rho_\iota\otimes|\det|^{q/2}\right).
\]
Using \eqref{eq:derivative-homology-realization}, we conclude that
\[
\pi^{(q)}
\cong
|\det|^{-q/2}
\otimes
\left(H_0(\mathfrak u_Q,\pi)\right)_{N_q,\psi_q}
\cong
\bigoplus_{\iota\in\mathscr I_q}\rho_\iota
=
\bigoplus_{\iota\in\mathscr I_q}\pi_\iota.
\]
The Chang--Cogdell decomposition is indexed by equivalence classes,
and the above bijection identifies those classes with $\mathscr I_q$.
Thus isomorphic summands arising from distinct indices are retained with
their multiplicities.
\end{proof}

The type~\(2\) exceptional poles arising in
\cite[Theorem~5.8]{Chai2015} are not shown there to be poles of
the corresponding \(L\)-functions. We instead prove the pole-order
comparison needed for Theorem~\ref{1.2} directly from the
factorization of archimedean \(L\)-functions.

\begin{thm}\label{order-characterization}
Let $\pi_1,\pi_2\in\operatorname{Irr}_{\mathrm{gen}}(G_n)$ be in
general position.

\begin{enumerate}[label=$(\it{\roman*})$]
    \item\label{order-characterization-1}
    Suppose that $s=s_0$ is a pole of
    $L(s,\pi_1\times\pi_2)$ of order $d$. Then $d\leq n$, and there
    exist an integer $q$ with
    $d\leq q\leq n$ and irreducible constituents
    \[
    \sigma_1\subseteq\pi_1^{(n-q)},
    \qquad
    \sigma_2\subseteq\pi_2^{(n-q)}
    \]
    such that $s=s_0$ is an exceptional pole for
    $(\sigma_1,\sigma_2)$ and
    \[
    \operatorname{pord}_{s_0}
    L(s,\sigma_1\times\sigma_2)=d.
    \]

    \item\label{order-characterization-2}
    Conversely, let $1\leq q\leq n$, and let
    \[
    \sigma_1\subseteq\pi_1^{(n-q)},
    \qquad
    \sigma_2\subseteq\pi_2^{(n-q)}
    \]
    be irreducible constituents. If $s=s_0$ is an exceptional pole
    for $(\sigma_1,\sigma_2)$ of order $d$, then
    \[
    \operatorname{pord}_{s_0}
    L(s,\pi_1\times\pi_2)\geq d.
    \]
\end{enumerate}
\end{thm}

\begin{proof}
For $i=1,2$, choose a finite index set $I_i$ and write
\[
\pi_i
=
\bigtimes_{\alpha\in I_i}\tau_{i,\alpha},
\]
where every $\tau_{i,\alpha}$ is either a character of $G_1$ or a
twisted discrete-series representation of $G_2$. Put
\[
\operatorname{rk}(\tau_{i,\alpha})
=
\begin{cases}
1,&\tau_{i,\alpha}\text{ is a character},\\
2,&\tau_{i,\alpha}\text{ is a twisted discrete series}.
\end{cases}
\]
Thus
\[
\sum_{\alpha\in I_i}\operatorname{rk}(\tau_{i,\alpha})=n.
\]

By the local Langlands correspondence \cite{Kna1994} and the multiplicativity of
Rankin–Selberg factors,
\begin{equation}\label{eq:block-factorization-L}
L(s,\pi_1\times\pi_2)
=
\prod_{\alpha\in I_1}
\prod_{\beta\in I_2}
L(s,\tau_{1,\alpha}\times\tau_{2,\beta}).
\end{equation}
The block factors are
\begin{align}
L\left(s,\chi_{\epsilon,x}\times\chi_{\eta,y}\right)
&=
\Gamma_{\mathbb R}
\left(s+x+y+\varepsilon_{\epsilon,\eta}\right),
\label{eq:L-char-char-general}
\\
L\left(s,\chi_{\epsilon,x}\times D_{k,t}\right)
&=
\Gamma_{\mathbb C}
\left(s+x+t+\frac{k-1}{2}\right),
\label{eq:L-char-discrete-general}
\\
L\left(s,D_{p,u}\times\chi_{\eta,y}\right)
&=
\Gamma_{\mathbb C}
\left(s+u+y+\frac{p-1}{2}\right),
\label{eq:L-discrete-char-general}
\\
L\left(s,D_{p,u}\times D_{k,t}\right)
&=
\Gamma_{\mathbb C}
\left(s+u+t+\frac{p+k}{2}-1\right)
\Gamma_{\mathbb C}
\left(s+u+t+\frac{|p-k|}{2}\right),
\label{eq:L-discrete-discrete-general}
\end{align}
where
\[
\varepsilon_{\epsilon,\eta}\in\{0,1\},
\qquad
\varepsilon_{\epsilon,\eta}\equiv\epsilon+\eta\pmod2.
\]

We first prove~\ref{order-characterization-1}. Define
\[
E(s_0)
=
\left\{
(\alpha,\beta)\in I_1\times I_2\,|\,
L(s,\tau_{1,\alpha}\times\tau_{2,\beta})
\text{ has a pole at }s=s_0
\right\}
\]
and, for $e=(\alpha,\beta)\in E(s_0)$, put
\[
d_e
=
\operatorname{pord}_{s_0}
L(s,\tau_{1,\alpha}\times\tau_{2,\beta}),
\qquad
r_e
=
\min\left\{
\operatorname{rk}(\tau_{1,\alpha}),
\operatorname{rk}(\tau_{2,\beta})
\right\}.
\]
From
\eqref{eq:L-char-char-general}--\eqref{eq:L-discrete-discrete-general},
\begin{equation}\label{eq:de-re-bound}
1\leq d_e\leq r_e.
\end{equation}

We next establish
\begin{equation}\label{eq:pole-pair-injectivity}
\begin{aligned}
(\alpha,\beta),(\alpha',\beta)\in E(s_0)
&\Longrightarrow \alpha=\alpha',\\
(\alpha,\beta),(\alpha,\beta')\in E(s_0)
&\Longrightarrow \beta=\beta'.
\end{aligned}
\end{equation}
Indeed, suppose that
$(\alpha,\beta),(\alpha',\beta)\in E(s_0)$. Subtracting the two pole
conditions obtained from
\eqref{eq:L-char-char-general}--\eqref{eq:L-discrete-discrete-general}
gives
\[
\begin{array}{c|c}
(\tau_{1,\alpha},\tau_{1,\alpha'})
&
\text{consequence}
\\ \hline
(\chi_{\epsilon,x},\chi_{\epsilon',x'})
&
x-x'\in\mathbb Z
\\[1mm]
(\chi_{\epsilon,x},D_{p,u})
&
x-u\in\tfrac12\mathbb Z
\\[1mm]
(D_{p,u},D_{p',u'})
&
u-u'\in\tfrac12\mathbb Z.
\end{array}
\]
Each conclusion contradicts the general-position hypothesis on
$\pi_1$, unless $\alpha=\alpha'$. The second implication in
\eqref{eq:pole-pair-injectivity} follows by interchanging $\pi_1$ and
$\pi_2$.

Since the Gamma functions have no zeros,
\begin{equation}\label{eq:d-sum-de}
d
=
\sum_{e\in E(s_0)}d_e.
\end{equation}
Set
\begin{equation}\label{eq:q-definition}
q
=
\sum_{e\in E(s_0)}r_e.
\end{equation}
By \eqref{eq:de-re-bound},
\[
d\leq q.
\]
Moreover, \eqref{eq:pole-pair-injectivity} implies that the projection
$E(s_0)\longrightarrow I_1$ is injective. Hence
\[
q
\leq
\sum_{(\alpha,\beta)\in E(s_0)}
\operatorname{rk}(\tau_{1,\alpha})
\leq
\sum_{\alpha\in I_1}
\operatorname{rk}(\tau_{1,\alpha})
=
n.
\]
Thus
\[
d\leq q\leq n.
\]

For every $e=(\alpha,\beta)\in E(s_0)$, we now define a pair $(\xi_{1,e},\xi_{2,e})$ of blocks of common rank $r_e$.

If $\tau_{1,\alpha}=\chi_{\epsilon,x}$ and $\tau_{2,\beta}=\chi_{\eta,y}$, define
\begin{equation}\label{eq:xi-char-char}
(\xi_{1,e},\xi_{2,e})
=
(\chi_{\epsilon,x},\chi_{\eta,y}).
\end{equation}

If $\tau_{1,\alpha}=\chi_{\epsilon,x}$ and $\tau_{2,\beta}=D_{k,t}$, put $h=\frac{k-1}{2}$ and $z_e=s_0+x+t+h$. Since the corresponding mixed factor has a pole, $z_e\in\mathbb Z_{\leq0}$. Let $\delta_e\in\{0,1\}$ be the unique parity satisfying $\varepsilon_{\epsilon,\delta_e}\equiv z_e\pmod2$, and define
\begin{equation}\label{eq:xi-char-discrete}
(\xi_{1,e},\xi_{2,e})
=
\left(
\chi_{\epsilon,x},
\chi_{\delta_e,t+h}
\right).
\end{equation}
Then $z_e+\varepsilon_{\epsilon,\delta_e}\in2\mathbb Z_{\leq0}$, so $L(s,\xi_{1,e}\times\xi_{2,e})$ has a simple pole at $s=s_0$.

Similarly, if $\tau_{1,\alpha}=D_{p,u}$ and $\tau_{2,\beta}=\chi_{\eta,y}$, put $h=\frac{p-1}{2}$ and $z_e=s_0+u+y+h$. Choose the unique $\delta_e\in\{0,1\}$ satisfying $\varepsilon_{\delta_e,\eta}\equiv z_e\pmod2$, and define
\begin{equation}\label{eq:xi-discrete-char}
(\xi_{1,e},\xi_{2,e})
=
\left(
\chi_{\delta_e,u+h},
\chi_{\eta,y}
\right).
\end{equation}

Finally, if $\tau_{1,\alpha}=D_{p,u}$ and $\tau_{2,\beta}=D_{k,t}$, define
\begin{equation}\label{eq:xi-discrete-discrete}
(\xi_{1,e},\xi_{2,e})
=
(D_{p,u},D_{k,t}).
\end{equation}

In all four cases,
\begin{equation}\label{eq:selected-block-order}
\operatorname{rk}(\xi_{1,e})
=
\operatorname{rk}(\xi_{2,e})
=
r_e
\end{equation}
and
\begin{equation}\label{eq:selected-pole-order}
\operatorname{pord}_{s_0}
L(s,\xi_{1,e}\times\xi_{2,e})
=
d_e.
\end{equation}
For the mixed cases, both sides of
\eqref{eq:selected-pole-order} are equal to $1$; for the other two
cases, the selected block factor is identical to the original block
factor.

Define
\begin{equation}\label{eq:sigma-selected}
\sigma_i
=
\bigtimes_{e\in E(s_0)}\xi_{i,e},
\qquad i=1,2,
\end{equation}
in any fixed order. By \eqref{eq:selected-block-order},
$\sigma_1$ and $\sigma_2$ are representations of $G_q$.

We verify that
\[
\sigma_i\subseteq\pi_i^{(n-q)}.
\]
For a character block of $\pi_i$, assign the parameter $a=0$ if its
index occurs in a pair in $E(s_0)$, and assign $a=1$ otherwise. For a
discrete-series block of $\pi_i$, assign
\[
b=
\begin{cases}
0,&\text{if it occurs in a discrete-series--discrete-series pair},\\
1,&\text{if it occurs in a mixed pair},\\
2,&\text{if its index occurs in no pair in }E(s_0).
\end{cases}
\]
When $b=1$, use the parity $\delta_e$ defined in
\eqref{eq:xi-char-discrete} or \eqref{eq:xi-discrete-char}.
The assignments are unambiguous by
\eqref{eq:pole-pair-injectivity}, and their total is
\[
n-\sum_{e\in E(s_0)}r_e=n-q.
\]
They therefore define an element of $\mathscr I_{n-q}$ in
Proposition~\ref{explicit-derivative}, and the corresponding
constituent is precisely $\sigma_i$. By Proposition~\ref{explicit-derivative}, \(\sigma_1\) and
\(\sigma_2\) are irreducible generic representations of \(G_q\) in
general position.

It remains to prove that $s=s_0$ is exceptional for
$(\sigma_1,\sigma_2)$. For $e\in E(s_0)$, the pair
$(\xi_{1,e},\xi_{2,e})$ is either a pair of characters or a pair of
twisted discrete-series representations.

Suppose first that $\xi_{1,e}=\chi_{\epsilon,x}$ and $\xi_{2,e}=\chi_{\eta,y}$. There exists $r\in\mathbb Z_{\geq0}$ such that $s_0+x+y+\varepsilon_{\epsilon,\eta}=-2r$. Set $a_e=2r+\varepsilon_{\epsilon,\eta}$. Then
\begin{equation}\label{eq:char-exceptional-data}
a_e\in\mathbb Z_{\geq0},
\qquad
\epsilon\equiv\eta+a_e\pmod2,
\qquad
x=-(y+a_e+s_0).
\end{equation}

Suppose next that $\xi_{1,e}=D_{p,u}$ and $\xi_{2,e}=D_{k,t}$. Set $A_e=s_0+u+t$ and $a_e=-2A_e$. By \eqref{eq:L-discrete-discrete-general}, either $A_e+\frac{|p-k|}{2}=-r$ or $A_e+\frac{p+k}{2}-1=-r$ for some $r\in\mathbb Z_{\geq0}$. Accordingly, $a_e=|p-k|+2r$ or $a_e=p+k-2+2r$. In both cases,
\begin{equation}\label{eq:discrete-exceptional-data}
a_e\in\mathbb Z_{\geq0},
\qquad
|p-k|\leq a_e,
\qquad
p\equiv k+a_e\pmod2,
\qquad
u=-\left(t+\frac{a_e}{2}+s_0\right).
\end{equation}

Put
\[
M=\sum_{e\in E(s_0)}a_e.
\]
After ordering the blocks of $\sigma_1$ and $\sigma_2$ according to
the same ordering of $E(s_0)$,
\eqref{eq:char-exceptional-data} and
\eqref{eq:discrete-exceptional-data} satisfy precisely the conditions
of Theorem~\ref{classification-general}. Hence $s=s_0$ is an
exceptional pole with level $M$ for $(\sigma_1,\sigma_2)$.

We finally compute its order. Let
\[
E_\sigma(s_0)
=
\left\{
(e,f)\in E(s_0)\times E(s_0)\,|\,
L(s,\xi_{1,e}\times\xi_{2,f})
\text{ has a pole at }s=s_0
\right\}.
\]
By \eqref{eq:selected-pole-order},
\[
(e,e)\in E_\sigma(s_0)
\qquad
\text{for every }e\in E(s_0).
\]
Applying
\eqref{eq:pole-pair-injectivity} to their inducing blocks gives
\[
E_\sigma(s_0)
=
\left\{
(e,e)\,|\,e\in E(s_0)
\right\}.
\]
Consequently,
\[
\begin{aligned}
\operatorname{pord}_{s_0}
L(s,\sigma_1\times\sigma_2)
=
\sum_{e\in E(s_0)}
\operatorname{pord}_{s_0}
L(s,\xi_{1,e}\times\xi_{2,e})
=
\sum_{e\in E(s_0)}d_e
=d,
\end{aligned}
\]
where the last equality is \eqref{eq:d-sum-de}. This proves
\ref{order-characterization-1}.

We proceed to~\ref{order-characterization-2}. Let
\[
\sigma_1\subseteq\pi_1^{(n-q)},
\qquad
\sigma_2\subseteq\pi_2^{(n-q)}
\]
be irreducible constituents, and suppose that $s=s_0$ is exceptional
for $(\sigma_1,\sigma_2)$ of order $d$.

Fix the two indices in Proposition~\ref{explicit-derivative} defining
$\sigma_1$ and $\sigma_2$. Every inducing block $\xi$ of $\sigma_i$
then has a unique source block, denoted by $\widetilde{\xi}$, among
the inducing blocks of $\pi_i$:
\[
\widetilde{\xi}
=
\begin{cases}
\chi_{\epsilon,x},
&\xi=\chi_{\epsilon,x}
\text{ comes from an original character block},\\[1mm]
D_{k,t},
&\xi=D_{k,t},\\[1mm]
D_{k,t},
&\xi=\chi_{\delta,t+\frac{k-1}{2}}
\text{ comes from splitting }D_{k,t}.
\end{cases}
\]
Distinct inducing blocks of $\sigma_i$ have distinct source blocks.

By Theorem~\ref{classification-general}, there is a finite index set
$J$ and decompositions
\[
\sigma_i
=
\bigtimes_{j\in J}\xi_{i,j},
\qquad i=1,2,
\]
such that, for every $j\in J$, the pair
$(\xi_{1,j},\xi_{2,j})$ consists either of two characters or of two
twisted discrete-series representations, and
\[
L(s,\xi_{1,j}\times\xi_{2,j})
\]
has a pole at $s=s_0$.

Applying \eqref{eq:pole-pair-injectivity} to $\sigma_1$ and $\sigma_2$
shows that the factors with indices $j\neq j'$ are holomorphic at
$s=s_0$. Therefore,
\begin{equation}\label{eq:derivative-diagonal-order}
d
=
\sum_{j\in J}
\operatorname{pord}_{s_0}
L(s,\xi_{1,j}\times\xi_{2,j}).
\end{equation}

We claim that, for every $j\in J$,
\begin{equation}\label{eq:source-factor-inequality}
\operatorname{pord}_{s_0}
L(s,\widetilde{\xi}_{1,j}\times\widetilde{\xi}_{2,j})
\geq
\operatorname{pord}_{s_0}
L(s,\xi_{1,j}\times\xi_{2,j}).
\end{equation}

If neither $\xi_{1,j}$ nor $\xi_{2,j}$ is obtained by splitting a
discrete-series block, the two local factors in
\eqref{eq:source-factor-inequality} are equal.

Suppose that exactly one of $\xi_{1,j}$ and $\xi_{2,j}$ is obtained
by splitting a discrete-series block. By symmetry, we may assume that
$\widetilde{\xi}_{1,j}=\xi_{1,j}$ and
\[
\xi_{1,j}=\chi_{\epsilon,x},
\qquad
\xi_{2,j}=\chi_{\delta,t+\frac{k-1}{2}},
\qquad
\widetilde{\xi}_{2,j}=D_{k,t}.
\]
Then
\begin{equation}\label{eq:mixed-factor-splitting}
L(s,\chi_{\epsilon,x}\times D_{k,t})
=
\prod_{\delta'=0}^{1}
L\left(
s,\chi_{\epsilon,x}\times
\chi_{\delta',t+\frac{k-1}{2}}
\right).
\end{equation}
Indeed, this is the identity
\[
\Gamma_{\mathbb C}(z)
=
\Gamma_{\mathbb R}(z)\Gamma_{\mathbb R}(z+1).
\]
Thus the factor involving $\xi_{2,j}$ is one of the two factors on the
right-hand side of \eqref{eq:mixed-factor-splitting}, proving
\eqref{eq:source-factor-inequality}. The case in which
$\xi_{1,j}$ comes from a split discrete-series block is identical.

Finally, suppose that
\[
\xi_{1,j}
=
\chi_{\delta,u+\frac{p-1}{2}},
\qquad
\xi_{2,j}
=
\chi_{\delta',t+\frac{k-1}{2}},
\]
with
\[
\widetilde{\xi}_{1,j}=D_{p,u},
\qquad
\widetilde{\xi}_{2,j}=D_{k,t}.
\]
Put
\[
z=s+u+t+\frac{p+k}{2}-1.
\]
Then
\[
L(s,\xi_{1,j}\times\xi_{2,j})
=
\Gamma_{\mathbb R}
\left(z+\varepsilon_{\delta,\delta'}\right),
\]
which is one of the two factors in
\[
\Gamma_{\mathbb C}(z)
=
\Gamma_{\mathbb R}(z)\Gamma_{\mathbb R}(z+1).
\]
Since $\Gamma_{\mathbb C}(z)$ is a factor of
\[
L(s,D_{p,u}\times D_{k,t}),
\]
the inequality \eqref{eq:source-factor-inequality} follows.

The source blocks attached to distinct $j\in J$ are distinct on both
sides. Hence the source block pairs
\[
(\widetilde{\xi}_{1,j},\widetilde{\xi}_{2,j}),
\qquad j\in J,
\]
give distinct factors of $L(s,\pi_1\times\pi_2)$. Using
\eqref{eq:block-factorization-L}, the absence of zeros of the Gamma
functions, \eqref{eq:source-factor-inequality}, and
\eqref{eq:derivative-diagonal-order}, we obtain
\[
\begin{aligned}
\operatorname{pord}_{s_0}
L(s,\pi_1\times\pi_2)
\geq
\sum_{j\in J}
\operatorname{pord}_{s_0}
L(s,\widetilde{\xi}_{1,j}\times
\widetilde{\xi}_{2,j})
\geq
\sum_{j\in J}
\operatorname{pord}_{s_0}
L(s,\xi_{1,j}\times\xi_{2,j})
=d.
\end{aligned}
\]
This proves~\ref{order-characterization-2}.
\end{proof}

For $0\leq q<n$, write
\[
\pi_1^{(q)}
=
\bigoplus_i\pi_{1,i}^{(q)},
\qquad
\pi_2^{(q)}
=
\bigoplus_j\pi_{2,j}^{(q)},
\]
where the direct sums are the indexed decompositions of
Proposition~\ref{explicit-derivative}.

\begin{corollary}\label{max-exceptional-order}
For every $s_0\in\mathbb C$,
\[
\operatorname{pord}_{s_0}
L(s,\pi_1\times\pi_2)
=
\max_{\substack{0\leq q<n\\ i,j}}
d_{\mathrm{ex}}
\left(
s_0;
\pi_{1,i}^{(q)},
\pi_{2,j}^{(q)}
\right).
\]
\end{corollary}

\begin{proof}
The inequality
\[
\operatorname{pord}_{s_0}
L(s,\pi_1\times\pi_2)
\geq
d_{\mathrm{ex}}
\left(
s_0;
\pi_{1,i}^{(q)},
\pi_{2,j}^{(q)}
\right)
\]
for every $q,i,j$ follows from
Theorem~\ref{order-characterization}-\ref{order-characterization-2}.

If $L(s,\pi_1\times\pi_2)$ is holomorphic at $s=s_0$, then
Theorem~\ref{order-characterization}-\ref{order-characterization-2}
shows that no pair of derivative constituents can have an exceptional
pole at $s_0$. Hence every term in the maximum is zero. If it has a pole of order $d$, then
Theorem~\ref{order-characterization}-\ref{order-characterization-1}
produces a pair of derivative constituents for which the exceptional
pole order is exactly $d$. Hence equality holds.
\end{proof}

We now prove Theorem~\ref{1.2}.

\begin{proof}[Proof of Theorem~\ref{1.2}]
Put
\[
H(s)
=
\underset{\substack{0\leq q<n\\ i,j}}
{\operatorname{l.c.m.}}
\left\{
L_{\mathrm{ex}}
\left(
s,\pi_{1,i}^{(q)}\times\pi_{2,j}^{(q)}
\right)^{-1}
\right\}.
\]
Fix $s_0\in\mathbb C$. By
Proposition~\ref{Lex-pole-orders}, for every $q,i,j$,
\[
\operatorname{ord}_{s_0}
L_{\mathrm{ex}}
\left(
s,\pi_{1,i}^{(q)}\times\pi_{2,j}^{(q)}
\right)^{-1}
=
d_{\mathrm{ex}}
\left(
s_0;
\pi_{1,i}^{(q)},
\pi_{2,j}^{(q)}
\right).
\]
Therefore, Proposition~\ref{prop:lcm-characterization} gives
\[
\operatorname{ord}_{s_0}H
=
\max_{\substack{0\leq q<n\\ i,j}}
d_{\mathrm{ex}}
\left(
s_0;
\pi_{1,i}^{(q)},
\pi_{2,j}^{(q)}
\right).
\]
By Corollary~\ref{max-exceptional-order}, the right-hand side is equal
to
\[
\operatorname{pord}_{s_0}
L(s,\pi_1\times\pi_2).
\]
Since the archimedean Gamma factors have no zeros, this is also equal
to
\[
\operatorname{ord}_{s_0}
L(s,\pi_1\times\pi_2)^{-1}.
\]
Consequently,
\[
\operatorname{ord}_{s_0}H
=
\operatorname{ord}_{s_0}
L(s,\pi_1\times\pi_2)^{-1}
\]
for every $s_0\in\mathbb C$.

By Definitions~\ref{Lex} and~\ref{lcm}, the function $H$ is a
normalized finite product of inverse $\Gamma_{\mathbb R}$-factors.
On the other hand, the factorization
\eqref{eq:block-factorization-L}, together with
\eqref{eq:L-char-char-general}--\eqref{eq:L-discrete-discrete-general}
and the identity
\[
\Gamma_{\mathbb C}(z)
=
\Gamma_{\mathbb R}(z)\Gamma_{\mathbb R}(z+1),
\]
shows that $L(s,\pi_1\times\pi_2)^{-1}$ is also a normalized finite
product of inverse $\Gamma_{\mathbb R}$-factors. The uniqueness
assertion in Proposition~\ref{prop:lcm-characterization} therefore
implies that
\[
H(s)=L(s,\pi_1\times\pi_2)^{-1}. \qedhere\]
\end{proof}

\begin{acknowledgements}
Work on this paper was undertaken when Y. Jo and S. Nadimpalli attended the conference ``SCPP Workshop on Automorphic forms and Representation Theory" at Indian Institute of Technology, Bombay. We would like to thank R. Raghunathan for organizing the conference and providing us with a good opportunity for the interaction. Y. Jo and S. Nadimpalli are indebted to the LMS-Sheffield Symposia for support and hospitality during the symposium ``The Langlands Programme: Recent Trends, New Developments, and Applications". 
Y. Jo thanks J. Cogdell for bringing to his attention an unfinished project \cite{CC1999} many years ago.
The authors express their sincere gratitude to U. K. Anandavardhanan, A. K. Mondal, and D. Prasad for their insightful comments and helpful discussion.
\end{acknowledgements}

\begin{funding}
Y. Jo and A. Yadav were partially supported by Global-Learning \& Academic research institution for Master’s·PhD students, and Postdocs(G-LAMP) Program of the National Research Foundation of Korea(NRF) grant funded by the Ministry of Education(No. RS-2025-25442252).
Y. Jo was also partially supported by the NRF grant 
funded by the Korean government(No. RS-2023-00209992). 
\end{funding}

\end{section}

\end{document}